\documentclass[journal]{IEEEtran}

\ifCLASSINFOpdf
\else
\usepackage[dvips]{graphicx}
\fi
\ifCLASSINFOpdf
 \usepackage[pdftex]{graphicx}
 \graphicspath{{./Figs/}}
 \DeclareGraphicsExtensions{.pdf,.jpeg,.png}
\else
 \usepackage[dvips]{graphicx}
 \graphicspath{{./Figs/}}
 \DeclareGraphicsExtensions{.eps}
\fi

\usepackage{cite}

\usepackage{amsmath}
\usepackage{array}

\ifCLASSOPTIONcompsoc
  \usepackage[caption=false,font=normalsize,labelfont=sf,textfont=sf]{subfig}
\else
  \usepackage[caption=false,font=footnotesize]{subfig}
\fi

\usepackage{url}

\usepackage{verbatim}

\usepackage{nonlinfilt} 

\usepackage{mathtools,amssymb,nccmath}

\newcommand{\cnvxcnvgres}[1]{C_1~\left( \dfrac{2L_{\max} -\mu_{\min}}{2L_{\max} +\mu_{\min}}\right)^{#1}} 
\newcommand{\singupdatebound}[1]{C_0 a^{#1}} 
\usepackage{letltxmacro}
\LetLtxMacro{\oldalgorithmic}{\algorithmic}
\LetLtxMacro{\endoldalgorithmic}{\endalgorithmic}
\renewenvironment{algorithmic}[1][0]{%
	\hrulefill\par
	\oldalgorithmic[#1]}
{\endoldalgorithmic\par
	\vspace*{-.5\baselineskip}
	\hrulefill\par
}

\begin{document}
	\title{Streaming Solutions for Time-Varying \\Optimization Problems}
	\author{Tomer Hamam, Justin Romberg, \IEEEmembership{Fellow, IEEE}
		\thanks{The authors are with the School of Electrical and Computer Engineering at Georgia Tech in Atlanta.  Email: tomer.hamam@gatech.edu, jrom@ece.gatech.edu.}
		\thanks{This work was supported by and ARL DCIST CRA W911NF-17-2-0181 and by C-BRIC, one of six centers in JUMP, a Semiconductor Research Corporation (SRC) program sponsored by DARPA.}
		\thanks{ Submitted October 30, 2021.}
			}

	\maketitle
\begin{abstract}
%
%
This paper studies streaming optimization problems that have objectives of the form
$ \sum_{t=1}^Tf(\vx_{t-1},\vx_t)$.   In particular, we are interested in how the solution $\hat\vx_{t|T}$ for the $t$th frame of variables changes as $T$ increases.  While incrementing $T$ and adding a new functional and a new set of variables does in general change the solution everywhere, we give conditions under which  $\xhat_{t|T}$ converges to a limit point $\vx^*_t$ at a linear rate as $T\rightarrow\infty$.  As a consequence, we are able to derive theoretical guarantees for algorithms with limited memory, showing that limiting the solution updates to  only a small number of frames in the past sacrifices almost nothing in accuracy.
We also present a new efficient Newton online algorithm (NOA), inspired by these results, that updates the solution with fixed per-iteration complexity of $ \ordof{3Bn^3}$, independent of $T$, where $B$ corresponds to how far in the past the variables are updated, and $n$ is the size of a single block-vector. Two streaming optimization examples, online reconstruction from non-uniform samples and inhomogeneous Poisson intensity estimation, support the theoretical results and show how the algorithm can be used in practice.
\end{abstract}	


\IEEEpeerreviewmaketitle
	

\section{Introduction}
\label{sec:intro}
\IEEEPARstart{W}{e} consider time-varying convex optimization problems of the form 
\begin{equation}
	\label{eq:def:general problem formulation}
	\minimize_{\xbar_T}J_T(\xbar_T) := \sum_{t=1}^{T} f_t(\vx_{t-1},\vx_{t} ), 
\end{equation}
where each $\vx_t$ is block variable in $\R^n$ and $\xbar_T=(\vx_0;\vx_1;\cdots;\vx_T)\in\R^{n(T+1)}$.
%
We are particularly interested in solving these programs {\em dynamically}; we study below how the solutions $\hat\xbar_T$ evolve as $T$ increases and how we can move from the minimizer of $J_T$ to the minimizer of $J_{T+1}$ in an efficient manner.

Optimization problems of the form \eqref{eq:def:general problem formulation} arise, broadly speaking, in applications where we are trying to estimate a time-varying quantity from data that is presented sequentially.  In signal processing, they are used for online least-squares \cite{jiang2004revisit, vahidi2005recursive} and estimation of sparse vectors \cite{li2007estimation, lewicki1999coding, asif2009dynamic,gruber1997statistical, angelosante2010online}. 
They have also been used for low-rank matrix recovery in recommendation systems \cite{xu2017simultaneous, NIPS2016_6449f44a}, audio restoration and enhancement \cite{ godsill2002digital, fong2002monte}, medical imagery applications \cite{sarkka2012dynamic, lin2006dynamic},
and inverse problems \cite{tarantola2005inverse,kaipio2006statistical}. Applications in other fields include online convex optimization \cite{hall2015online, cesa2012new, langford2009sparse}, adaptive learning \cite{andrieu2001rao, hartikainen2010kalman}, time series prediction \cite{sarkka2004time,sarkka2007cats}, and optimal control \cite{stengel1994optimal,maybeck1982stochastic}. 
Closely related problems also come from estimation algorithms in wireless sensor networks \cite{doherty2001convex, shi2010distributed}, multi-task learning \cite{hall2015online}, and in simultaneous localization and mapping (SLAM) and pose graph optimization (PGO) \cite{zou2012coslam, kaess2012isam2, shan2020orcvio,carlone2016planar}.

%

Even though each of the functions in \eqref{eq:def:general problem formulation} only depends on two block variables, the reliance of $f_{t}$ on both $\vx_t$ and $\vx_{t-1}$ couples all of the functionals in the sum.  When a new ``frame'' is added to the program above, meaning $T\rightarrow T+1$, a new term is added to the sum in the functional, and a new set of variables is introduced.  This new term will affect the optimality of all of the $\vx_t$.

The most well-known example of \eqref{eq:def:general problem formulation} is when the $f_t$ are least-squares losses on linear functions of $\vx_{t-1}$ and $\vx_t$, 
\begin{equation}
\label{eq:linearls}
	f_t(\vx_{t-1},\vx_t) = \|\mB_{t}\vx_{t-1} +\mA_t\vx_t-\vy_t\|^2.
\end{equation}
In this case, \eqref{eq:def:general problem formulation} has the same mathematical structure as the Kalman filter \cite{kalman1960new,rauch1965maximum, aravkin2013kalman}  and can be solved with a streaming least-squares algorithm.
If we use $\hat{\vx}_{t|T}$ to denote the optimal estimate of $\vx_t$ when $J_T$ is minimized, then $\hat\vx_{T+1|T+1}$ can be computed by applying an appropriate affine function to $\hat\vx_{T|T}$, and then the $\hat\vx_{T-j|T+1}, j=1,\ldots,T$ are computed recursively (by again applying affine functions) with a backward sweep through the variables.  
This updating algorithm, which requires updating each frame only once, follows from the fact that the system of equations for solving $J_{T+1}$ has a {\em block-tridiagonal} structure; a block LU factorization can be computed on the fly and the $\hat\vx_{T-j|T}$ are then computed using back substitution.

When the $f_t$ have any other form than \eqref{eq:linearls},  moving from the solution of $J_T$ to $J_{T+1}$ is much more complicated.  
Unlike the linear least-squares case, we cannot update the solutions with a single backward sweep.
%
%
However, as we describe in detail in Section~\ref{sec:cnvx case} below, \eqref{eq:def:general problem formulation} retains a key piece of structure from the least-squares case. The Hessian matrix of $J_T$ has the same block-tridiagonal structure as the system matrix corresponding to \eqref{eq:linearls}.
%


Our main mathematical contribution shows that when the Hessian matrix exhibits a kind of block diagonal dominance in addition to the tridiagonal structure, the solution vectors are only weakly coupled in that adding a new term to \eqref{eq:def:general problem formulation} does not significantly affect the solutions far in the past. 

Theorem~\ref{thm:ls convergence} in Section~\ref{sec:ls} and Theorem~\ref{thm:convergence of the updates} in Section~\ref{sec:cnvx case} below show that the difference between $\hat{\vx}_{t|T+1}$ and $ \hat{\vx}_{t|T}$ decreases exponentially in $T-t$.  As a result, $\hat{\vx}_{t|T}$ and $\hat{\vx}_{t|T+1}$ will not be too different for even moderate $T-t$. 

The correction terms' rapid convergence gives rise to the \textit{optimization filtering} approach described in the second half of Section \ref{sec:cnvx case}.  We show how we can (approximately) solve \eqref{eq:def:general problem formulation} with bounded memory by only updating a relatively small number of the $\vx_t$ in the past each time a new loss function is added.  We show that under appropriate conditions, the error due to this ``truncation'' of early terms does not accumulate as $T$ grows, meaning that the online algorithm is stable.  Theorem \ref{thm:cnvxbufferrbnd} gives these sufficient conditions and bounds the error as a function of the memory in the system.

The remainder of the paper is organized as follows. We briefly overview related work in the existing literature in Section \ref{sec:related work}. 
In Section \ref{sec:ls}, we study the particular case of least-squares loss \eqref{eq:linearls}, and give conditions under which the $\hat\vx_{t|T}$ converge rapidly as $T$ increases.  This allows us to control the error of a standard recursive least-squares algorithm when the updates are truncated.
Section~\ref{sec:cnvx case} extends these results where the $f_t$ are general (smooth and strongly) convex loss functions, with an online-like Newton-type algorithm for solving these problems presented in Section~\ref{sec:nwton alg}.
Numerical examples to support our theoretical results are given in Section \ref{sec:numerical example}.
Proofs are deferred to the appendices. Notation follows the standard convention. 


\subsection{Related work}
\label{sec:related work}


As more problems in science and engineering are being posed in the language of convex optimization, several new frameworks have been introduced that move away from batch solvers in static setting to optimization problems that change in time.  A recent survey \cite{simonetto2020time} details one such framework.  There the goal is to reconstruct a solution trajectory  
\begin{equation}
    \label{eq:timevarying}
    \vx^*(t) = \argmin_{\vx\in\R^n} f(\vx;t),
\end{equation}
given an ensemble of objective functions $f(\cdot;t)$ indexed by time $t$.  Research on problems of the type \eqref{eq:timevarying}, see in particular \cite{koppel2015target, fazlyab2017prediction, simonetto2016class}, is concerned with fast algorithms that take advantage of structural properties of the ensemble $f(\cdot;t)$ (e.g.\ that it varies slowly in time) to approximate $\hat\vx^*(t)$ in a manner that is much more efficient than if the optimization program was treated in isolation.  The goal is to achieve something like real-time tracking of the solution as $f(\cdot;t)$ evolves with $t$.

This framework is fundamentally different than the one in \eqref{eq:def:general problem formulation} that is studied in this paper.  In \eqref{eq:timevarying}, there is a single optimization variable, and at a fixed time $t$ the function $f(\cdot;t)$ provides all the information needed to compute the ideal optimal solution $x^*(t)$.  In contrast, the formulation in \eqref{eq:def:general problem formulation} adds a new optimization variable along with a new function as $T$ increases.  The solution $\hat{\vx}_{t|T}$ of \eqref{eq:def:general problem formulation} for frame $t$ changes as $T$ increases, and a key point of interest is the structural properties of the $f_t$ that result in the solutions settling quickly.  This is a question that is completely independent of any algorithm used to solve (or approximate the solution of) \eqref{eq:def:general problem formulation}.  This concept that the functions in the future affect the optimal solutions in the past does not exist for \eqref{eq:timevarying}.


Another formulation for time-varying optimization problems is known as \emph{online convex optimization} (OCO) \cite{zinkevich2003online, shalev2012online, hazan2016introduction}.  In a typical OCO problem, a player makes a prediction $\vx_t$, then suffers a loss $f_t(\vx_t)$.  The process repeats at the next time step, with the player (hopefully) improving their predictions as they adapt to the structure in the problem.  The effectiveness of an algorithm is measured using a regret function, a typical example of which is 
\[
	\mathrm{Regret}_T = \sum_{t=1}^Tf_t(\vx_t) - \min_{\vx}\sum_{t=1}^Tf_t(\vx).
\]     
This is, again, a very different framework than the one we are considering in  \eqref{eq:def:general problem formulation}. As mentioned above, we are introducing new optimization variables as the time index $T$ increases, and the fact that $f_t$ shares variables with $f_{t-1}$ and $f_{t+1}$ explicitly provides the dependence between the solutions.  Moreover, our Theorems~\ref{thm:convergence of the updates} and \ref{thm:bondcnvgrad} below are not tied to particular algorithms, they are concerned with how the solutions to  \eqref{eq:def:general problem formulation} change as $T$ increases.
Our analysis does not involve any notion of prediction, but the concept of ``smoothing'' (updating the variables at previous time steps) plays a central role.
Updating previous solutions as new information is revealed is not emphasized in the OCO literature.

The OCO framework has also been extended to more flexible regret models appropriate for dynamic environments; different types of adaptive regret are considered in \cite{hazan2007logarithmic,hall2013online}. 
These models do model optimization problems that themselves vary in time and whose solutions may be  coupled in some sense. 
However, these problems are again about making improved predictions, and no notion of smoothing exists.

Finally, and as we mentioned above, if the $f_t$ are least-squares functionals, our framework is equivalent to the celebrated Kalman filter \cite{kalman1960new} with backward smoothing \cite{bryson1963smoothing, rauch1965maximum}.
Convergence results in the literature primarily assume a generative stochastic model and are concerned with the convergence of the covariance of the $\vx_t$ and how it results in the asymptotic stability of the filter.  In \cite{bougerol1993kalman}, the authors show  convergence of the estimate error covariance by showing that the transformation that propagates this covariance across time steps is a contraction.  

Our main results in this context give conditions under which the smoothing updates converge at a linear rate.  While we do not know of comparable existing results in the literature, there is some qualitative discussion related to this convergence in \cite[Ch. 7]{moore1979optimal} that starts to draw a relation between the accuracy of the smoothing to the asymptotic behavior of the gain matrix and covariance matrix in terms of their eigenvalue structure.
The recent work \cite{aravkin2013kalman, aravkin2021algorithms} also treats Kalman smoothing as solving a block tridiagonal system, analogous to our approach to the least-squares setting in Section~\ref{sec:ls}, and show how the spectral properties of submatrices lead to numerical stability of the computations.
The work \cite{cao2003exponential} does show linear convergence of the estimate $\vx_{t|t}$ but only in the special case where the state is static.

The extended Kalman filter (EKF) \cite{cox1964estimation,sage1971estimation} is the classic way to extend this online tracking framework to problems that are nonlinear.  It works by linearizing the problem around the current estimate (thus approximating the original problem), then performing the updates the same way as in linear least-squares.  Our nonlinear framework is focused on solving the optimization program directly.

As a concluding remark, we want to re-emphasize that not only that our filtering optimization framework different from previous work on time-varying optimization problems, our main theoretical results are on intrinsic properties of the optimization problem and how its solution changes over time. Our Corollary~\ref{corr:ls:buffer error bound} and Theorem~\ref{thm:cnvxbufferrbnd}  provide bounds on the approximation error with truncated updates that are algorithm independent.  They can be interpreted as a bridge for existing algorithms to leverage these results; the truncated streaming least-squares algorithm in Section~\ref{sec:ls} and the truncated Newton algorithm in Section~\ref{sec:nwton alg} are two examples that do exactly that.


\section{Streaming Least-Squares } 
\label{sec:ls}
\noindent
In this section, we look at solving \eqref{eq:def:general problem formulation} in the particular case where the cost functions are regularized linear least-squares terms similar to \eqref{eq:linearls}:
\footnote{
Throughout $ \| \cdot \|$ refers to the $Euclidean$ norm vectors and its induced norm for matrices .
}
\begin{equation}\label{eq:flstik}
	f_t(\vx_{t-1},\vx_t) = \| \mB_t \vx_{t-1} + \mA_t \vx_t -\vy_t\|^2 + \gamma\|\vx_t\|^2.
\end{equation}
To ease the notation, we will fix the sizes of the variables $\vx_t\in\R^{n}$ and matrices $\mA_t,\mB_t\in\R^{m\times n}$ to be the same, but everything below is easy to generalize.  An important application, and one which is discussed further in Section~\ref{sec:numerical example} below, is streaming reconstruction from non-uniform samples of signals that are being represented by local overlapping basis elements.
%
There is an elegant way to solve this type of streaming least-squares problem that is essentially equivalent to the Kalman filter.  We will quickly review how this solution comes about below, as it will allow us to draw parallels to the case when the $f_t$ are general convex functions.

We begin by presenting a matrix formulation of the problem that expresses the minimizer as the solution to an inverse problem.
This formulation exposes the problem’s nice structure, leading to an efficient forward-backward LU factorization solver. 
We then discuss the stability of the factorization and show that it also leads to convergence of the updates (Theorem~\ref{thm:ls convergence}). Lastly, as a corollary of Theorem~\ref{thm:ls convergence}, we show that the updates can be truncated, allowing the streaming solution to operate with finite memory, with very little additional error.

\subsection{Matrix formulation}
\label{subsec:ls:matrix form} 
\noindent
With the $f_t$ as in \eqref{eq:flstik}, the optimization problem \eqref{eq:def:general problem formulation} is equivalent to solving a structured linear inverse problem.  At frame $T$, we are estimating the right-hand side of
\begin{equation*}\label{eq:PhiKT}
	\medmath{
	\underbrace{\begin{bmatrix}
		\vy_0 \\ \vy_1 \\ \vy_2 \\ \vy_3 \\ \vy_4 \\ \vdots \\ \vy_{T}
	\end{bmatrix}}_{\bvy_T}
	}
	=
	\medmath{
	\underbrace{\begin{bmatrix}
		\mA_0 & \mzero & \cdots & & & & \mzero \\
		\mB_1 & \mA_1 & \mzero & \cdots & & & \mzero\\
		\mzero & \mB_2 & \mA_2 & \mzero & \cdots & & \mzero\\
		\mzero & \mzero & \mB_3 & \mA_3 & \mzero & \cdots & \mzero \\
		\mzero & \mzero & \mzero & \mB_4 & \mA_4 & \cdots & \mzero \\
		\vdots & & & & \ddots & \ddots & \vdots \\
		\mzero & \cdots & & & \cdots & \mB_{T} & \mA_{T}
	\end{bmatrix}}_{\bmPhi_T}
	}
 	\medmath{
 	\underbrace{\begin{bmatrix} 
		\vx_0 \\ \vx_1 \\ \vx_2 \\ \vx_3 \\ \vx_4 \\ \vdots \\ \vx_{T}
	\end{bmatrix}}_{\bvx_T}
	}
	+ \medmath{\mathrm{noise}},
\end{equation*}

with a least-squares loss.  The minimizer of this (regularized) least-squares problem satisfies the normal equations
\begin{equation}
	\label{eq:normal eqs} 
	\hat{\bvx}_T = (\bmPhi^\T_T\bmPhi_T +\gamma\mId )^{-1}\bmPhi^\T\bvy_T.
\end{equation}
When $T$ is large, solving \eqref{eq:normal eqs} can be expensive or even infeasible. However, the structure of $\mPhibar_T$ (it has only two nonzero block diagonals) allows an efficient method for updating the solution as $T\rightarrow T+1$.

\subsection{Block-tridiagonal systems} 
\noindent
The key piece of structure that our analysis and algorithms take advantage of is that the system matrix $\bmPhi^\T_T\bmPhi_T +\gamma\mId$ in \eqref{eq:normal eqs} is block-tridiagonal.  In this section, we overview how this type of system can be solved recursively and introduce conditions that guarantee that these computations are stable.  We will use these results both to analyze the least-squares case and the more general convex case, where the Hessian matrix has the same block-tridiagonal structure.

Consider a general block-tridiagonal system of equations
\begin{equation}
    \label{eq:blktridiagsys}
    \begin{bmatrix}
	\mH_0 & \mE_0^\T & & & \\
	\mE_0 & \mH_1 & \mE_1^\T & & \\
	& \ddots & \ddots & \ddots & \\
	& & \mE_{T-2} & \mH_{T-1}& \mE_{T-1}^\T \\
	& & &			\mE_{T-1} & \mH_T
	\end{bmatrix}
	\begin{bmatrix} 
	    \vx_0 \\ \vx_1 \\ \vdots \\ \vx_{T-1} \\ \vx_T
	\end{bmatrix}
	=
	\begin{bmatrix} 
	    \vg_0 \\ \vg_1 \\ \vdots \\ \vg_{T-1} \\ \vg_T
	\end{bmatrix}
\end{equation}
There is a standard numerical linear algebra technique for calculating the (block) LU factorization of a block-banded matrix (see \cite[Chapter~4.5]{golub2012matrix}); the matrix above gets factored as
\begin{equation*}
	\label{eq:ls:LU fact}
	\medmath{
	\underbrace{
	\begin{bmatrix} 
		\mQ_0 & \mzero & \cdots & & \mzero \\ 
		\mE_0 & \mQ_1 & \mzero & & \\
		\mzero & \mE_1 & \mQ_2 & \ddots & \\ 
		\vdots & & \ddots & \ddots & \mzero \\ 
		\mzero & \cdots & \mzero & \mE_{T-1} & \mQ_T
	\end{bmatrix}
	}_{\mLbar_T} 
	\underbrace{
	\begin{bmatrix} 
		\mId & \mU_0 & \mzero & & \\
		\mzero & \mId & \mU_1 & \mzero & \\ 
		\vdots & & \ddots & \ddots & \\ & & & \ddots & \mU_{T-1} \\ 
		\mzero & & & \mzero & \mId 
	\end{bmatrix}
	}_{\mUbar_T} 
	}.
\end{equation*} 
The $\mE_t$ in \eqref{eq:ls:LU fact} are the same as in \eqref{eq:blktridiagsys} while the $\mQ_t$ and $\mU_t$ can be computed recursively using $\mQ_0 = \mH_0$, and then for $t=1,\ldots,T$
\begin{align}
	\mU_{t-1} &= \mQ_{t-1}^{-1}\mE_{t-1}^\T, \nonumber \\
	\mQ_t &= \mH_{t} - \mE_{t-1}\mU_{t-1} = \mH_{t} - \mE_{t-1}\mQ_{t-1}^{-1}\mE_{t-1}^\T. \label{eq:Qt}
\end{align}
This, in turn, gives us an efficient way to solve the system in \eqref{eq:blktridiagsys} using a forward-backward sweep: we start by initializing the ``forward variable'' as $\vv_0 = \mQ_0^{-1}\vg_0$, then move forward; for $t=1,\ldots,T$ we compute
\begin{equation}
    \label{eq:forwardsweep}
    \vv_t = \mQ_t^{-1}(\vg_t - \mE_{t-1}\vv_{t-1}).
\end{equation}
After computing $\vv_T$, we hold $\vvbar = \mLbar^{-1}\vgbar$ in our hands.  To compute $\hat\vxbar = \mUbar^{-1}\vvbar$,  
we first set $\hat\vx_T = \vv_T$ and then move backward; for $t=T-1,\ldots,0$ we compute
\begin{equation}
    \label{eq:backwardsweep}
    \hat\vx_t = \vv_t - \mU_t\hat\vx_{t+1}.
\end{equation}

This algorithm relies on the $\mQ_t$ being invertible for all $t$.  As these are defined recursively, it might, in general, be hard to determine their invertibility without actually computing them.  The lemma below, however, shows that if the system in \eqref{eq:blktridiagsys} is block diagonally dominant, in that the $\mE_t$ are smaller than the $\mH_t$, then well-conditioned $\mH_t$ will result in well-conditioned $\mQ_t$, and as a result, the reconstruction process above is well-defined and stable.
%
%
%
\begin{lemma}
	\label{lm:Qcond}
	Suppose that there exists a $\kappa\geq 1$ and $\delta, \theta$ such that for all $t\geq 0$, $\|\kappa^{-1}\mH_t-\mId\|\leq\delta$ and $\|\mE_t\|\leq\kappa\theta$  with $\delta < 1$ and $\theta \leq (1-\delta)/2$.  Then for all $t\geq 0$,
	\[
		\|\kappa^{-1}\mQ_t-\mId\|\leq\eps_*, \quad \eps_* = \frac{1+\delta}{2} - \sqrt{\frac{(1-\delta)^2}{4}-\theta^2} < 1.
	\]
\end{lemma}
\begin{proof}
	From the recursion that defines the $\mQ_t$ in \eqref{eq:Qt} we have $\|\kappa^{-1}\mQ_t-\mId\|\leq \epsilon_t$, where $\{\epsilon_t\}$ is a sequence that obeys
	\[
		\epsilon_0=\delta, \quad \epsilon_t = \delta + \frac{\theta^2}{1-\epsilon_{t-1}}.
	\]
	For the $\theta,\delta$ in the lemma statement, these $\epsilon_t$ form a monotonically nondecreasing sequence that converges to $\eps_*$.
\end{proof}

One of the main consequences of Lemma~\ref{lm:Qcond}, and a result that we will use for our least-squares and more general convex analysis, is that a uniform bound on the size of the blocks $\|\vg_t\|$ on the right-hand side of \eqref{eq:blktridiagsys} implies a uniform bound on the size of the blocks in the solution $\|\hat\vx_t\|$.

\begin{lemma}
	\label{lm:xtbound}
	Suppose that the conditions of Lemma~\ref{lm:Qcond} hold.  Suppose also that we have a uniform bound on the norm of each of the blocks of $\bvg$, $\|\vg_t\|\leq M$.  Then if $\rho = \theta/(1-\eps_*) < 1$,
	\[
		\|\xhat_{t|T}\|~\leq~\frac{M(1-\rho^{T-t+1})}{(1-\eps_*)(1-\rho)^2} ~\leq~ \frac{M}{(1-\eps_*)(1-\rho)^2}.
	\]
\end{lemma}

The proof of Lemma~\ref{lm:xtbound}, which we present in Appendix~\ref{sec:xtboundproof}, essentially just traces the steps through the forward-backward sweep making judicious use of the triangle inequality.  Later on, we will see that our ability to  bound the solution on a frame-by-frame basis plays a key role in showing that the streaming solutions to \eqref{eq:def:general problem formulation} converge as $T$ increases.

From general linear algebra, we have the following result that will be of use later on.
\begin{lemma}\label{lem:sparseb3dsystake2}
	Consider the $(m+n) \times (m+n)$ system 
	\begin{equation*}\label{eq:lemblocksysy}
		\left[ \begin{array}{c|ccc}
			\mA &   & \mV^T &  \\ \hline
			& \ & \     & \ \\
			\mV   & \ & \mB   &   \\
			\     & \ & \     &
		\end{array}
		\right]
		\begin{bmatrix}
			\vh_0\\
			\hline 
			\\
			\vy \\
			\\
		\end{bmatrix} 
		=
		\begin{bmatrix}
			\vq_0\\
			\hline 
			\\
			\mzero \\
			\\
		\end{bmatrix},
	\end{equation*}
	with matrices $ \mA\in \R^{m\times m} $, $\mB\in \R^{n\times n} $, $ \mV\in \R^{n\times m} $, and vectors $\vh_0,~\vq_0\in R^m  $ and $ \vy=(\vy_1,\dots,\vy_n)\in\R^n $.	
	Suppose that $ \|\mV\| \leq \alpha $, and $\mB$ is nonsingular with $\|\mB^{-1}\|\leq \beta$. Then
	$
	\|\vy\| \leq {\alpha}{\beta}\|\vh_0\|
	$, and in particular, 
	$ 	\|\vy_i\|\leq {\alpha}{\beta}\|\vh_0\|.$ 
\end{lemma}

Lemma \ref{lem:sparseb3dsystake2} shows there is a simple relation between the first and remaining elements of the solution for the particular right-hand side above.
\subsection{Streaming solutions} 
\label{subs:ls:sol update}
\noindent
The forward-backward algorithm for solving block-tridiagonal systems can immediately be converted into a streaming solver for problems of the form \eqref{eq:flstik}.  Indeed, the algorithm that does this, detailed explicitly as Algorithm~\ref{alg:Dynamic strm LS} below, mirrors the steps in the classic Kalman filter exactly (with the backtracking updates akin to ``smoothing'').

For fixed $T$, the least-squares problem in \eqref{eq:flstik} amounts to a tridiagonal solve as in \eqref{eq:blktridiagsys} with 
\begin{align*}
	&\mE_t    = \mA_{t+1}^\T\mB_{t+1},\\
	&\mH_{t}  =  \mA_t^\T\mA_t + \mB_{t+1}^\T\mB_{t+1} + \gamma\mId,  \quad t=0,\dots,T-1 ~,\\
	&\mH_{T}=	\mA_T^\T\mA_T  +\gamma\mId,\\
	& \vg_t = \mA_t \vy_t + \mB_{t+1}\vy_{t+1}, \qquad t=0,\dots,T-1,\\
	& \vg_T =\mA_T \vy_T
	.
\end{align*}
If we have computed the solution $\hat\vxbar_T = \{\hat\vx_{t|T}\}_{t=0}^T$, then when $T$ is incremented, we can move to the new solution $\hat\vxbar_{T+1}$ by updating the $\mQ_T$ matrix, computing the new terms $\mE_{t-1},\mU_{t-1}$ in the block LU factorization (thus completing the forward sweep in one additional step), computing $\hat\vx_{T+1|T+1}$, then sweeping backward to update the solution by computing the $\hat\vx_{t|T+1}$ for $t=T,\ldots,0$.

\begin{algorithm}\caption{Streaming Least-squares } 
	\label{alg:Dynamic strm LS}
	\alglanguage{pseudocode} 
	\begin{algorithmic}
		\State $[\vy_0,\mA_0]\gets \mathrm{GetSampleBatch}(0)$
		\State $\mQ_0' \gets \mA_0^\T\mA_0 + \gamma\mId ~~ $
		\State $\vg_0' \gets \mA_0^\T\vy_0$
		\State $\hat\vx_0 \gets \mQ_0^{'-1}\vg_0'$
		\For{$t=1,2,\ldots$}
		\State $[\vy_t,\mA_t,\mB_t]\gets\mathrm{GetSampleBatch}(t)$
		\State $\mQ_{t-1}\gets \mQ_{t-1}' + \mB_t^\T\mB_t  $
		\State $\vg_{t-1}\gets \vg_{t-1}' + \mB_t^\T\vy_t$
		\State $\vv_{t-1}\gets \mQ_{t-1}^{-1}(\vg_{t-1} - \mE_{t-2}\vv_{t-2})$
		\State $\mE_{t-1}\gets \mA_t^\T\mB_t$
		\State $\mU_{t-1}\gets \mQ_{t-1}^{-1}\mE_{t-1}^\T$
		\State $\mQ_t'\gets \mA_t^\T\mA_t - \mE_{t-1}\mU_{t-1} + \delta\mId$
		\State $\vg_t'\gets \mA_t^\T\vy_t$
		\State $\hat\vx_{t|t}\gets \mQ_t^{'-1}(\vg_t'-\mE_t\vv_t)$
		\For{$\ell=1,\ldots,t$ }
		\State $\hat\vx_{t-\ell|t}\gets \vv_{t-\ell} - \mU_{t-\ell}\hat\vx_{t-\ell+1|t}$
		\EndFor
		\EndFor
	\end{algorithmic}
\end{algorithm}
We now have the natural questions: under what conditions does an  $\vx_t^\star$  exist such that $\hat\vx_{t|T}\rightarrow\vx_t^\star$, and if one exists, how fast do the solutions converge? 
Our first theorem provides one possible answer to these questions.  At a high level, it says that while the solution in every frame changes as we go from $t\rightarrow T+1$, this effect is {\em local} if the block-tridiagonal system is block diagonally dominant.  Well-conditioned $\mH_t$ and relatively small $\mE_t$ result in rapid (linear) convergence of $\hat\vx_{t|T}\rightarrow\vx_t^\star$.

\begin{theorem}
\label{thm:ls convergence}
Suppose that the $\mH_t$ and $\mE_t$ generated by Algorithm~\ref{alg:Dynamic strm LS} obey the conditions of Lemma~\ref{lm:Qcond} and $\theta< 1-\epsilon_*$.  Suppose also that the size of the $\vy_t$ are bounded as
\[
    M_y = \sup_{t\geq 0} \left\|\begin{bmatrix} \vy_t \\ \vy_{t+1} \end{bmatrix}\right\|. 
\] 
Then there exists $\vx_{t}^*$ such that 
\begin{equation*} \label{eqn:ls asym sol}
	\hat\vx_{t|T} \rightarrow  \vx_{t}^* \quad\text{as}\quad t\leq T\rightarrow\infty,
\end{equation*} 
and there is a constant $C(\epsilon_*,\theta,\delta)$ such that
\begin{equation*}\label{eq: cnvx: thm upd conv res}
	 \|\hat\vx_{t|T} -  \vx_{t}^*\| ~\leq~ C(\epsilon_*,\theta,\delta) ~
M_y ~ \left(\frac{\theta}{1-\epsilon_*}\right)^{T-t} 
\end{equation*}
 for all $  t\leq T $.
\end{theorem}

The purpose of $\kappa$ in the theorem statement above is to make the result scale-invariant; a natural choice is to take $\kappa$ as the average of the largest and smallest eigenvalues of the matrices along the main block diagonal. 

Finally, we note that the condition that $ \theta<1-\eps_*$ is closely related to asking the system matrix to be (strictly) block diagonally dominant\footnote{We are referring to the block diagonal dominance defined in \cite[Definition 1]{feingold1962block} as $\lambda_{\mathrm{min}}(\mA_{j,j}) \geq \sum_{\substack{k=1 \\ k \neq j}}^{N}\left\|A_{j, k}\right\|$, where $\lambda_{\mathrm{min}}(\cdot)$ returns the smallest eigenvalue.}. We could guarantee block diagonal dominance by asking the smallest eigenvalue of each $\mH_t$ is larger than $\|\mE_{t-1}\|+\|\mE_t\|$, which we can ensure by taking $ \theta/(1-\delta) < 1/2$.
In the next theorem, we show that the exponential convergence of the least-squares estimate in Theorem~\ref{thm:ls convergence} allows us to stop updating $\hat\vx_{t|T}$ once $T-t$ gets large enough with almost no loss in accuracy.

\subsection{Truncating the updates}
\label{subsec:ls:trunc strm}
\noindent
The result of Theorem~\ref{thm:ls convergence} that the updates exhibit exponential convergence suggests that we might be able to ``prune'' the updates by terminating the backtracking step early.  In this section, we formalize this by bounding the additional error if we limit ourselves to a buffer of size $B$.  This allows the algorithm to truly run online, as the memory and computational requirements remain bounded (proportional to $B$).


To produce the truncated result, we use a simple modification of Algorithm~\ref{alg:Dynamic strm LS}.  The forward sweep remains the same (and exact), while the backtracking loop (the inner 'for' loop at the end) only updates the $B$ most recent frames, stopping at $\ell=B-1$:

\alglanguage{pseudocode} 
\begin{algorithmic}[0]
    	\For{$\ell=1,\ldots,B-1$ } 
		\State $\hat\vz_{t-\ell|t}\gets \vv_{t-\ell} - \mU_{t-\ell}\hat\vz_{t-\ell+1|t}$
		\EndFor
\end{algorithmic}
To avoid confusion, we have written the truncated solutions as $\hat\vz_{t|t'}$, and we will use $\vz_t^\star$ to be the ``final'' value $\vz_t^\star = \hat\vz_{t+1|t+B}$.  Note that to perform the truncated update, we only have to store the matrices $\mU_t$ and the vectors $\vv_t$ for $B$ steps in the past.  The schematic diagram in Figure~\ref{fig:streamingwithbuffer} illustrates the architecture and dynamical flow of the algorithm.

\begin{figure}[!tbh]
	\centering
	\includegraphics[width=1.0\linewidth]{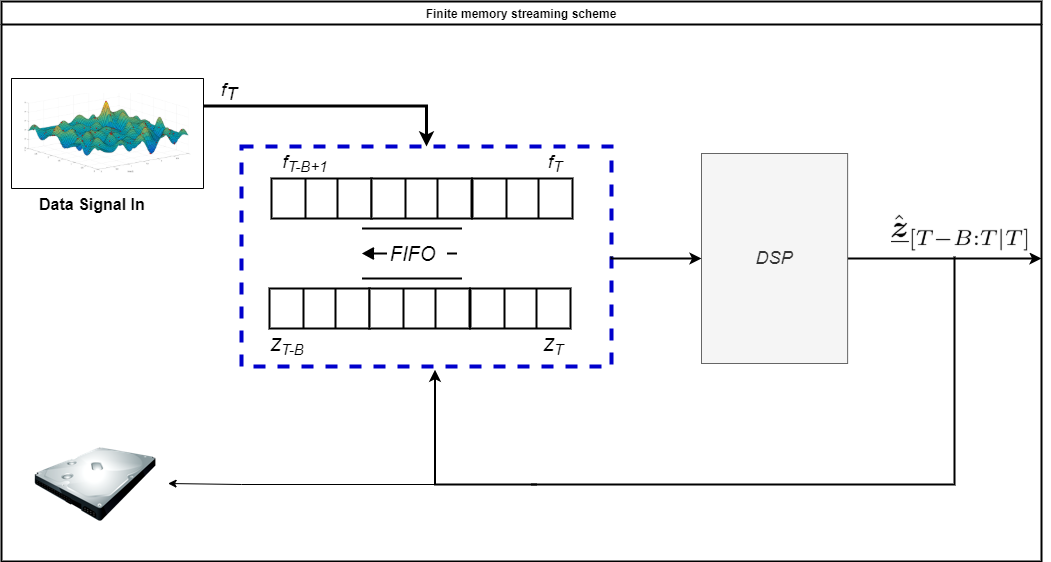}
	\caption{Schematic diagram of streaming with finite memory data architecture: The buffers pertain to fast memory which is assumed limited and hence used only to store the last $B$ loss functions and the latest corresponding estimated solution variables,
		$\{ \xfhat_{t|T} \}_t, t=T-B,\dots ,T$. After each solve, we have feedback to the buffer and as a possible implementation,  $\xfhat_{T-B|T}$ is offloaded to peripheral memory to keep track of the entire solution's trajectory. }
	\label{fig:streamingwithbuffer}
\end{figure}

The following corollary shows that the effect of the buffer is mild: the error in the final answer decreases exponentially in buffer size.


\begin{corollary} 
    \label{corr:ls:buffer error bound} 
	Let $ \{\vx_t^*\}$ denote the sequences of untruncated asymptotic solutions, and  $\{\vz^*_t\} $ the final truncated solutions for a buffer size $B$. Under the conditions of Theorem~\ref{thm:ls convergence}, we have
	\begin{equation*}
	\label{eq:ls truncation err bnd}
	 \| \vx_t^* - \xfstar_t \| \leq  C(\epsilon_*,\theta,\delta) 
	 M_y  \left(\frac{\theta}{1-\epsilon_*}\right)^{B} 
	\end{equation*}
	for all $0\leq t\leq T-B $.
\end{corollary}

The corollary is established simply by realizing that $\hat\vz_{t|T}=\hat\vx_{t|T}$ for $T=t,\ldots,t+B$ and then applying Theorem~\ref{thm:ls convergence}.
Section \ref{subsec:lvl cros} demonstrates that for a practical application, excellent accuracy can be achieved with a modest buffer size.  There we look at the problem of signal reconstruction from non-uniform level crossings.

Our truncated least-squares algorithm is functionally the same  as the fixed-lag Kalman smoother \cite{moore1973fixed} and its variants \cite{aravkin2021algorithms}.  Corollary~\ref{corr:ls:buffer error bound} gives us principled guidance on the choice of window size for these algorithms.  It applies to streaming least-squares problems in other contexts; see in particular the online signal reconstruction from non-uniform samples example in Section~\ref{subsec:lvl cros}.


\section{Streaming with convex cost functions}
\label{sec:cnvx case}
\noindent

In this section, we consider solving \eqref{eq:def:general problem formulation} in the more general case where the $f_t$ are convex functions.  In studying the least-squares case above, we saw that the coupling between the variables means that adding a term to \eqref{eq:def:general problem formulation} (increasing $T$) requires an update of the entire solution.  We also saw that if the least-squares system in \eqref{eq:blktridiagsys} is block diagonally dominant, implying in some sense that coupling between the $\vx_t$ is ``loose,''  then these updates are essentially local in that the magnitudes of the updates decay exponentially as they backpropagate.  Below, we will see that a similar effect occurs for general smooth and strongly convex $f_t$.


We present our main theoretical result--- the linear convergence of the updates--- in two steps. Theorem~\ref{thm:convergence of the updates} shows that if it is possible to find an initialization point with uniformly bounded gradients, then we have the same exponential decay in the updates to the $\hat{x}_t$ as $t$ moves backward away from $T$.  Theorem~\ref{thm:convergence of the updates} shows that we can guarantee such initialization points when the Hessian of the aggregated objective is block diagonal dominant. 

\subsection{Convergence of the streaming solution}
\noindent

Let $ \mathcal{S}^{2,1}_{\mu,L}(\mathcal{D}) $ denote the class of $\mu$-strongly convex functions with two continuous derivatives, with their first derivatives $ L $-Lipschitz continuous on $ \mathcal{D}$.  

\begin{assumption}\label{ass:strong convexity and lip grad}
	For all $ t\geq 1 $, we assume that:
	\begin{enumerate}
		\item $f_t\in\mathcal{S}^{2 , 1}_{\mu_t, L_t}(\R^n\times\R^n)$;
		\item $L_t$ and $ \mu_t$ are uniformly bounded, 
		\[0<\mu_{\min} \leq \mu_{t} \leq L_{t} \leq L_{\max} < \infty.\]
		
	\end{enumerate}
\end{assumption}
This assumption is equivalent to a uniform bound on the eigenvalues of the Hessians of the $f_t$; for every $t\geq 0$ and every $\vx,\vy\in\mathcal{D}$ we have
\[
    \mu_t\leq \lambda_{\mathrm{min}}(\nabla^2f_t(\vx,\vy)) 
    \leq\lambda_{\mathrm{max}}(\nabla^2f_t(\vx,\vy))\leq L_t.
\]
The following lemma translates the uniform bounds on the convexity and Lipschitz constants of the $f_t$ to similar bounds on the aggregate function.

\begin{lemma}
	\label{lemma:lip and conv of partially separable sum}
	If $ f_t(\vx_{t-1},\vx_t)\in \mathcal{S}^{2,1}_{\mu_t,L_t}(\R^n\times \R^n)$, then $ J_T=\sum_{t=1}^Tf_t  \in \mathcal{S}^{2,1}_{\tilde\mu,\tilde L}(\R^{n}\times\dots\times\R^n)$, 
	for $\tilde L = {2 L}_{\max}$ and $ \tilde\mu =  \mu_{\min}$.	
\end{lemma}

Lemma \ref{lemma:lip and conv of partially separable sum} above is enough to establish the convergence of the updates provided that we can assume uniform bounds on the gradients of the initialization points. 
\begin{theorem}
    \label{thm:convergence of the updates}
    Suppose that the $f_t$ are smooth and strongly convex as in Assumption~\ref{ass:strong convexity and lip grad}. Suppose also that there exists a constant $M_g$ and a set of initialization points $\{\vw_T\}$ such that 
    \begin{equation}
        \label{eq:unigradbound}
        \|\nabla f_T(\xhat_{T-1|T-1},\vw_T)\|\leq M_g,
    \end{equation}
    for all $T>0$.  Then there exists $\vx_{t}^*$ such that 
	\begin{equation*} \label{eqn:cnvx asym sol}
		\hat\vx_{t|T} \rightarrow  \vx_{t}^* \quad\text{as}\quad t\leq T\rightarrow\infty,
	\end{equation*} 
	and a constant $C_1$ that depends on $\mu_{\min}, L_{\max}, M_g$ such that
	\begin{equation}
	    \label{eq:linearconvxt}
		\|\hat\vx_{t|T} - \vx_{t}^*\|
		~\leq~ 	
		C_1~\left( \dfrac{2L_{\max} -\mu_{\min}}{2L_{\max} +\mu_{\min}}\right)^{T-t}.  
	\end{equation}
\end{theorem}
We prove Theorem~\ref{thm:convergence of the updates} in Appendix~\ref{sec:weaker conv result proof}.  The argument works by tracing the steps the gradient descent algorithm takes when minimizing $J_T$ after being initialized at the minimizer of $J_{T-1}$ (with the newly introduced variables initialized at $\vw_T$ satisfying \eqref{eq:unigradbound}). 

The condition \eqref{eq:unigradbound} is slightly unsatisfying as it relies on properties of the $f_t$ around the global solutions.  We will see in Theorem~\ref{thm:bondcnvgrad} below how to remove this condition by adding additional assumptions on the intrinsic structure of the $f_t$ and their relationships.  Nevertheless, \eqref{eq:unigradbound} does not seem unreasonable.  For example, if the solutions $\{\hat\vx_{T|T}\}$ are shown to be uniformly bounded, then we could use the fact that the $f_t$ have Lipschitz gradients to establish \eqref{eq:unigradbound} through an appropriate choice of $\vw_T$.

Although our theory lets us consider any $\vw_T$ , there are two very natural choices in practice.
One option is to simply minimize $f_T$ with the first variable fixed at the previous solution,

\begin{equation}\label{eq:initnewvar}
	\vw_T = \arg\min_{\vw} f_T(\xhat_{T-1|T-1} ,\vw).
\end{equation}

Alternatively, $\vw_T$ could be fixed a priori at the $\bar\vx_{T|T}$ that minimizes $f_t$ in isolation (see the definition in \eqref{eq:def:local sol} below).





\subsection{Boundedness through block diagonal dominance}
\label{subsec:Restricted conv}
\noindent
By adding some additional assumptions on the structure of the problem, we can guarantee condition \eqref{eq:unigradbound} in Theorem~\ref{thm:convergence of the updates}.  Our first (very mild) additional assumption is that the minimizers of the individual $f_t$, when computed in isolation, are bounded.
\begin{assumption}
    \label{ass:bounded local sol variation}
	With the minimizers of the isolated $f_t$ denoted as 
	\begin{equation}\label{eq:def:local sol}
		(\bar\vx_{t-1|t},~\bar\vx_{t|t}) 
		:= \argmin_{\vx_{t-1},\vx_t} f_t(\vx_{t-1},\vx_t),
	\end{equation}
	we assume that there exists a constant $M_x$ such that for all $ t\geq 0 $
	\begin{equation}
	    \label{eq:isolatedbound}
		\left\| \begin{bmatrix}
			\bar\vx_{t-1|t}\\
			\bar\vx_{t|t}
		\end{bmatrix}\right\| \leq M_x.
	\end{equation}
\end{assumption}
The strong convexity of the $f_t $ implies that these isolated minimizers are unique.

Note that there are two local solutions for the variables $\vx_t$ at time $t$: $\bar\vx_{t|t}$ is the second argument for the minimizer for $f_t$, while $\bar\vx_{t|t+1}$ is the first argument for the minimizer of $f_{t+1}$.  Assumption~\ref{ass:bounded local sol variation} also implies that these are close: $\|\bar\vx_{t|t} -  \bar\vx_{t|t+1}\| \leq 2M_x$.  We emphasize that Assumption~\ref{ass:bounded local sol variation} only prescribes structure on the minimizers of the individual $f_t$, not on the minimizers of the aggregate $J_T$.  
%
%
Showing how this assumed (but very reasonable) bound on the size of the minimizers of the individual $f_t$ translates into a bound on the size of the minimizers $\hat\vx_{t|T}$ of $J_T$ is a large part of Theorem~\ref{thm:bondcnvgrad}.

The last piece of structure that allows us to make this link comes from the Hessian of the objective $J_T$.  The fact that $f_t$ depends only on variables $\vx_{t-1}$ and $\vx_t$ means that the  Hessian is block diagonal, similar to the system matrix \eqref{eq:blktridiagsys} in the least-squares case:
\begin{equation}
\label{eq:hessianblk3diag}
\medmath{ 
	    \nabla^2 J_T =  
		\begin{bmatrix}
			\mH_0 & \mE_0^\T & & & \\
			\mE_0 & \mH_1 & \mE_1^\T & & \\\\
			& \ddots & \ddots & \ddots & \\\\
			& & \mE_{T-2} & \mH_{T-1}& \mE_{T-1}^\T \\
			& & &			\mE_{T-1} & \mH_T'
	\end{bmatrix}},
\end{equation}
where the main diagonal terms are given by\footnote{We use $ \nabla_{i,j}(\cdot)$ for ${\partial^2 (\cdot) }/{\partial_{\vx_i}\partial_{\vx_j}} $.}
\begin{align*}
	& \mH_t = 
	\begin{cases}
		\nabla_{0,0} f_1(\vx_{0},\vx_1) ,	& t=0;\\
		\nabla_{t,t}(f_t(\vx_{t-1},\vx_t) + f_{t+1}(\vx_t,\vx_{t+1}),	& 1\leq t< T;\\
		\nabla_{T,T} f_T(\vx_{T-1},\vx_T) ,	& t = T;\\
	\end{cases}
	\intertext{and the off diagonal terms by }
	&\mE_{t} =  \nabla_{t+1,t} f_{t+1}(\vx_{t},\vx_{t+1}) , \hspace{2.0cm} t=0,\dots, T-1.
\end{align*}

If the Hessian is block diagonally dominant everywhere, then we can leverage the boundedness of the isolated solutions \eqref{eq:isolatedbound} to show the boundedness of the aggregate solutions $\hat\vx_{t|T}$.

\begin{theorem} 
\label{thm:bondcnvgrad}
	Suppose that $ f_t$ as in Assumption \ref{ass:strong convexity and lip grad} and let
	\[
	 \kappa := \frac{2L_{\max}+\mu_{\min}}{2}, \quad \delta := \frac{{2L_{\max}-\mu_{\min}}}{{2L_{\max}+\mu_{\min}}},
	\]	
	 and suppose that for another constant $\theta>0$ it holds that 
	 \[
	  \|\mE_t(\bvx)\|\leq\kappa\theta, \quad \forall t\geq 0 .
	 \]
	 
	\noindent If it holds that 
	$
	\theta < {(1-\delta)}/{2}
	$, 
	and in addition, the isolated minimizers $\{(\bar\vx_{t|t},\bar\vx_{t|t+1})\}_{t=0}^{T}$ are bounded as in Assumption~\ref{ass:bounded local sol variation}, then the minimizers $\{\hat\vx_{t|T}\}_{t=0}^T$ of $J_T$ will be bounded as
	\begin{equation}
	    \label{eq:solutionbound}
	    \|\hat\vx_{t|T}\| ~\leq ~ M_g\frac{(1-\rho^{T-t+1})}{(1-\eps_*)(1-\rho)^2},
	\end{equation}
	where 
	\[
		M_g = 2M_x\kappa\sqrt{L_{max}^2+\theta^2},
	\]
	for some $ \eps_* ,\rho$ with $\rho=\theta/(1-\eps_*)<1$.
	
\end{theorem}
%

The bound \eqref{eq:solutionbound} on the size of the solutions can be used to bound the size of the gradient in \eqref{eq:unigradbound}.  If we initialize the variables in frame $T$ as the isolated minimizer from \eqref{eq:def:local sol}, $\vw_T = \bar\vx_{T|T}$, we have
\begin{align*}
    &\|\nabla f_T(\hat\vx_{T-1|T-1},\vw_T)\| = \\
    & \quad \|\nabla f_T(\hat\vx_{T-1|T-1},\bar\vx_{T|T}) - \nabla f_T(\bar\vx_{T-1|T},\bar\vx_{T|T})\| \\
    &\quad \leq L_{\mathrm{max}}\|\hat\vx_{T-1|T-1} - \bar\vx_{T-1|T}\| \\
    &\quad \leq L_{\mathrm{max}}\left(\|\hat\vx_{T-1|T-1}\| + \|\bar\vx_{T-1|T}\| \right).
\end{align*}
These two terms can then be bounded by \eqref{eq:solutionbound} and \eqref{eq:isolatedbound}.  Thus the conditions of Theorem~\ref{thm:bondcnvgrad} ensure the rapid convergence in \eqref{eq:linearconvxt}.

As in the least-squares case, the $\kappa$ in Theorem~\ref{thm:bondcnvgrad} is just a scaling constant so that the eigenvalues  of the $\kappa^{-1}\mH_t$ are within $1\pm\delta$.  Perhaps more important is the condition that $ \theta < (1-\delta)/2$.  We can interpret this condition as meaning that the coupling between the $f_t$ is ``loose''; in a second-order approximation to $J_T$, the quadratic interactions between the $\vx_t$ and themselves is stronger than between the $\vx_{t-1}$ or $\vx_{t+1}$.
As in the least-squares case, if the Hessian is strictly block diagonal dominant,
then again $(1-\delta) > 2\theta$.


In the least-squares case, the fast convergence of the updates enabled efficient computations through truncation of the backward updates. As we show next, truncation is also effective in the general convex case. 

\subsection{Streaming with finite updates}
\label{subsec: cnvx trunc nw }
\noindent
Just as we did in the least squares case, we can control the error in the convex case when the updates are truncated.  While our goals parallel those in Section~\ref{subsec:ls:trunc strm}, the analysis is more delicate.  The main difference in the general convex case is that when we stop updating a set of variables $\hat\vx_{t|T}$ it introduces (small) errors in the future estimates $\hat\vx_{t'|T'}$ for $t'>t, T'\geq T$; we were able to avoid this in the least-squares case since the forward ''state'' variables $\vv_t$ carry all the information needed to compute future estimates optimally.  Nonetheless, we show that the errors introduced by truncation in the convex case are manageable.

A key realization for this is that the problem \eqref{eq:def:general problem formulation} has a property similar to conditional independence in Markov random processes\footnote{We also note that if $\vx_1,\ldots,\vx_T$ were a Markov sequence of Gaussian random vectors, then the covariance matrix of the sequence would have the same tridiagonal structure as the Hessian in \eqref{eq:hessianblk3diag}.}.  
If we fix the variables $\vx_\tau$ in frame $\tau$, then the optimization program can be decoupled into two independent programs that minimize $\sum_{t=1}^\tau f_t(\vx_{t-1},\vx_t)$ and $\sum_{t=\tau+1}^Tf_t(\vx_{t-1},\vx_t)$ independently.

In particular, if we consider the truncation of the objective functional so that keeps its last $ B $ loss terms,
\begin{equation}\label{eq:JT-B}
	J_{T_B}(\vx_{T-B},\dots, \vx_{T}) = \sum_{t=T-B+1}^T f_t(\vx_{t-1}, \vx_t),
\end{equation}
fixing the first frame of variables to correspond to those in the minimizer of $J_T$, meaning $\vx_{T-B}=\hat{\vx}_{T_B|T}$, will allow us to recover the optimal $\vx_{T-B+1},\ldots,\vx_T$ by solving the smaller optimization problem of minimizing $J_{T_B}$.  We state this precisely with the following proposition.

\begin{proposition}
	\label{prop:solvWfixVar}	
	Let $\{\vy_{T-B+i}^* \}_{i=0}^B$  be the solution to 
\begin{align*}
		 & \minimize J_{T_B}(\vy_{T-B},\dots,\vy_{T}) \\
		 & \quad\quad\quad \text{subject to} ~~\vy_{T-B} = \xhat_{T-B|T}.
\end{align*}
Then the solution satisfies
	\[
		\vy_{T-B+i}^* = \xhat_{T-B+i|T}, ~~ i=1,\dots, B.
	\]
\end{proposition}


As before, we denote the truncated solutions as $\hat\vz_{t|t'}$ and their final values $\vz_t^* = \hat\vz_{t|T}$ for $T> t+B$.  

In the spirit of Proposition \ref{prop:solvWfixVar}, at each time step $T$, we fix $\vz_{T-B}=\vz^*_{T-B}$ as a boundary condition, then minimize the last $B$ terms in the sum in  \eqref{eq:def:general problem formulation}, setting $ t'=T-B+1$:
\begin{equation}\label{eq:buff opt prblm}
	\underset{(\vz_{t'},\dots,\vz_{T})}{\text{minimize}} f_{t'}(\xfstar_{t'-1},\vz_{t'}) 
	+ \sum_{t=t'+1}^T  f_t(\vz_{t-1},\vz_t).
\end{equation}


It should be clear by Proposition \ref{prop:solvWfixVar} and  the problem formulation in \eqref{eq:buff opt prblm}, that any difference between $\xhat_{T-B+i|T}$ and $\xfhat_{T-B+i|T}$ depends entirely on the difference of $\xhat_{T-B|T}$ and $\xfstar_{T-B}$.

\begin{lemma}\label{lem:truncdepbndtake2}
	Suppose the conditions of Theorem \ref{thm:bondcnvgrad} hold. 
	Then,
	\begin{equation*}
		\left\|\xhat_{T-B+i|T} - \xfhat_{T-B+i|T}  \right\| \leq  \frac{ \theta}{(1-\delta)}	\left\|\xhat_{T-B|T} - \xfstar_{T-B}  \right\|, 
	\end{equation*}	
	for $i=1,\dots,B$.
\end{lemma}
We give the proof of the Lemma in Appendix \ref{sec:lemtruncupdateproof}. 
Theorem \ref{thm:cnvxbufferrbnd} below builds on the results of the Lemma to establish error bound between $ \{\vx_t^*\} $ and $ \{\xfstar_t \}$ as a function of $ B $. 
\begin{theorem}	
	\label{thm:cnvxbufferrbnd}	
	Let $ \{\vx_t^*\}$ denote the sequences of untruncated asymptotic solutions, and  $\{\vz^*_t\} $ the final truncated solutions for a buffer size $B$. 
	Under the conditions of Theorem \ref{thm:bondcnvgrad}, we have
	\begin{equation*}\label{eq:buff err UB}
		\|\vx_t^* - \xfstar_t\|  \leq 
		C_b \left(\dfrac{2L_{\max} -\mu_{\min}}{2L_{\max} +\mu_{\min}}\right)^{B}
	\end{equation*}
	for some positive constant $ C_b(\mu_{\min},L_{\max} )$.  
\end{theorem}

Theorem \ref{thm:cnvxbufferrbnd} shows that, again, the truncation error shrinks exponentially as we increase the buffer size, where the shrinkage factor depends on the convex conditioning of the problem. 
In the next section, we leverage this result to derive an online truncated Newton algorithm.

\section{Newton algorithm}
\label{sec:nwton alg}
\noindent
In this section, we present an efficient truncated Newton online algorithm (NOA). 
The main result enabling this derivation stems from the bound given in Theorem \ref{thm:cnvxbufferrbnd} in Section \ref{sec:cnvx case}, and from the Hessian's block-tridiagonal structure. It implies that we can compute each Newton step using the same forward and then backward sweep we derived for the least-squares. Moreover, Theorem \ref{thm:cnvxbufferrbnd} implies that the solve can be implemented with finite memory with moderate computational cost that does not change in time. Hence, our proposed algorithm sidesteps the otherwise prohibiting complexity hurdle associated with Newton method while maintaining its favorable quadratic convergence.

\subsection{Setting up with truncated updates}
\noindent
\begin{assumption}\label{ass:gen NW ass}
	$ \nabla^2  J_T $ is $ M$-Lipschitz continuous.
\end{assumption}
Assumption \ref{ass:gen NW ass} is standard when considering Newton’s method. Combined with the assumptions stated in Section \ref{sec:cnvx case} we have sufficient conditions to guarantee the (local) quadratic convergence of Newton's algorithm (see, e.g., \cite{ortega2000iterative}). For smooth convex functions, the first-order optimality condition implies the solution to the nonlinear system $	\nabla  J_T(\xbarhat_T) = 0$ is also the optimal solution to \eqref{eq:def:general problem formulation}, hence solving \eqref{eq:buff opt prblm} is equivalent to solving 
\begin{equation}\label{eq:trunc nw def}
 F (\xfhat_{T-B+1},\dots,\xfhat_{T} ) = \mzero,
\end{equation}
where 
\[ F(\xfhat_{T-B+1},\dots,\xfhat_{T} ) = \nabla J_{T_B}(\xfstar_{T-B}, \xf_{T-B+1},\dots,\xf_{T} ).\]

\subsection{The Newton Online Algorithm (NOA)}
\noindent
To avoid confusion with the batch estimates, we use  $ \{\ybar\nw{k}\}$ to denote the sequence of updates obtained using the Newton solutions. The three main blocks of the algorithm:system update, initialization, and computation of the Newton step are described next.

\subsubsection{Newton system}
The Newton method solves the system in \eqref{eq:trunc nw def} iteratively. Starting with initial guess $ \ybar\nw{0} $, at every iteration, we update with
\begin{alignat*}{2}
	&\sbar\nw{k}\hspace{4mm} = - \mF'(\ybar\nw{k})^{-1}   F(\ybar\nw{k}),
	 \label{eq:nw update} \\
	&\ybar\nw{k+1}  = \ybar\nw{k} + \tau\nw{k}\sbar\nw{k},
\end{alignat*}
$ \tau\nw{m} \in (0,1]$ controls the step, , ensuring global convergence.

\subsubsection{Initialization}
The smoothness of the objective implies a warm-start is a reasonable choice. Hence, we initialize as
\begin{equation}\label{eq:nw init}
	\bvy\nw{0}_t := \xfhat_{t|T-1},  \qquad 0\leq t\leq T-1,
\end{equation} 
where for the new block-variable,$\vy_{T}\nw{0}$, we recommend using one of the two options proposed in   \eqref{eq:initnewvar} or \eqref{eq:def:local sol}.

\subsubsection{Computing the Newton step}

$ \mF' $ has the same block-tridiagonal structure as in the LS, meaning we can compute the Newton step with the same LU forward-backward solver derived in Section \ref{sec:ls}. However, a big difference is, for nonlinear systems, the factorized $ \mQ_t $ and $ \mU_t $ blocks are no longer stationary; updating $ \ybar\nw{k}$ updates $ \mF' $ and hence updates the factorization blocks. In other words, except for the first step where the change in $ \mF' $ is still local, we need to compute a new LU-factorization for every Newton step. 
\footnote{ For $ \sbar\nw{0} $, the initialization means only the last LU blocks change so we can carry most blocks  from the previous solve.}

Although every step involves recomputing the factorization, the exponential convergence of the updates means that we can get away with good accuracy even for a small B, which, together with the sparsity of the Hessian, implies a very reasonable per-iteration complexity of the order $ \ordof{3Bn^3}$, independent of $T$, and overall complexity of  $ \ordof{3Bn^3 k_{nw}}$, where $k_{nw} $ is the total number of Newton iterations required.
\iftrue

\begin{algorithm}[!t]
	\caption{Newton online algorithm (NOA) }
	\label{alg:NOA}
	\alglanguage{pseudocode}
	\algblock{If}{EndIf}
	\algcblock[If]{If}{ElsIf}{EndIf}
	\algcblock[Ife]{If}{Else}{EndIf}
	\begin{algorithmic}[1]
		\For {$T=1,2,\dots \quad  $}  \COMMENT{For each time step}
		\State Get new loss term $f_T(\cdot )$
		\vspace{1mm}
		\State Initialize $ \ybar\nw{0}$ \COMMENT{ see \eqref{eq:nw init} \label{ONW:init}}
		\vspace{1mm}
		\State \textit{Update Hessian factorization}:\\
		\hspace{0.8cm} $F_{T-1}, \mH_{T-1}, \mH_{T} ,\mE_{T-1}.$
		\vspace{1mm}
		\State $  k\gets 0  $ 
		\While{$ \norm{\nabla f(\vy\nw{k})}^2  \geq \epsilon_0$}:  \COMMENT{Newton's iteration}
		\vspace{1mm}
		\State Set $ t_s = 
		\begin{cases}
			T-1, & k=0\\
			T-B+1, & \mathrm{otherwise}  
		\end{cases} $
		\vspace{1mm}
		\State $ \mQ_{t_s} \gets \mH_{t_s}$  ;  $ \vv_{t_s} \gets \mQ_{t_s}^{-1} F_{t_s}$  \COMMENT{fwd}
		\For{$ t=t_s+1,\cdots ,T$}:  
		\State $ \mU_{t-1}\gets  \mQ_{t-1}^{-1}\mE_{t-1}^T$ ; $ \mQ_t \gets \mH_{t}-\mE_{t-1} \mU_{t-1} $ 
		\State $ \vv_t \gets  \mQ_t^{-1}\left( F_{t} - \mE_{t-1} \vv_{t-1} \right) $	
		\EndFor	
		
		\vspace{1mm}
		\State  $ \vs_{T} \gets \vv_{T} $ \COMMENT{bwd}
		\For{$ t=T-1, T-2,\cdots ,T-B+1$}
		\State $ \vs_{t} \gets \vv_t - \mU_t \vs_{t+1}$	
		\EndFor \label{NOA:LU ends}
		\vspace{1mm}
		\State $ \ybar\nw{k+1} \gets \ybar\nw{k} +  \tau\nw{k} \sbar $ \COMMENT{Newton step}
		\vspace{1mm}
		\State Reevaluate $[ F(\ybar\nw{k+1}) , \mF'(\ybar\nw{k+1})]$
		\State $ k\gets k+1 $
		\EndWhile
		\vspace{1mm}
		\State $ \xbarhat_{T} \gets \ybar\nw{k} $  \COMMENT{update global solution}
		\EndFor		
	\end{algorithmic}
\end{algorithm}
\fi
\section{Numerical examples}
\label{sec:numerical example}
\noindent
In this section we consider two working examples to showcase the proposed algorithms. In the first example, we use Algorithm \ref{alg:Dynamic strm LS} from the least-squares section for streaming reconstruction of a signal from its level crossing samples, and in the second example, we use NOA from Section \ref{sec:nwton alg} to efficiently solve neural spiking data regression.


\subsection{Online signal reconstruction from non-uniform samples}
\label{subsec:lvl cros}
\noindent
Inspired by the recent interest in level-crossing analog-to-digital converters (ADCs), we consider the problem of constructing a signal from its level crossings.  Rather than sample a continuous-time signal $x(t)$ on a set of uniformly spaced times, level-crossing ADCs output the times that $x(t)$ crosses one of a predetermined set of levels.  The result is a stream of samples taken at non-uniform (and signal-dependent) locations.  An illustration is shown in Figure~\ref{fig:3figs}.
%
The particulars of the experiment are as follows.  A bandlimited signal was randomly generated by super-imposing sinc functions (with the appropriate widths) spaced $1/64$ apart on the interval $[-5,21]$; the Sinc functions' heights were drawn from a standard normal distribution.  The level crossings were computed for $L=16$ different levels equally spaced between $[-2.5,2.5)$ in the time interval $[-0.25,16.25]$.  This produced $4677$ samples.  

The signal was reconstructed using the \textit{lapped orthogonal transform} (LOT). We applied the LOT to cosine IV basis functions resulting in a set of $ 16 $ frames of orthonormal basis bundles, each with $N=75$ basis functions and transition width $\eta=1/4$. 
A single sample $x(t_m)$ in batch $k$ (so $t_m\in\setT_k$) can then be written in terms of the expansion coefficients in frame bundles $k-1$ and $k$ as 
\begin{align}
\nonumber x(t_m) &= \sum_{n=1}^N\vx_{k-1,n}\psi_{k-1,n}(t_m) + \sum_{n=1}^N\vx_{k,n}\psi_{k,n}(t_m) \\\nonumber
\label{eq:sampleab} &= \<\vx_{k-1},\vb_m\> + \<\vx_k,\va_m\>, 
\end{align} 
where
$\vx_{k-1},\vx_k\in\R^N$ are the coefficient vectors (across all $N$ components) in bundles $k-1$
and $k$, and $\va_m,\vb_m\in\R^N$ are samples of the basis functions at $ t_m $(and are independent of the
actual signal $x(t)$).
Collecting the corresponding measurement vectors $\va_k,\vb_k$ for all $M_k=|\setM_k|$ samples in batch $k$ together as rows gives the $M_k\times N$ matrices $\mA_k$ and $\mB_k$ from  \eqref{eq:flstik} , and we can use Algorithm \ref{alg:Dynamic strm LS} for the signal reconstructed using . \\

The results in Figure \ref{fig:3figs} show the reconstructed signal at three consecutive time frames. Table \ref{tab:lagerror} gives a more detailed accounting of the reconstruction error for different buffer lengths.
We can see that a buffer of three frames  already yields seven digits of accuracy.  Hence using $B=3$ in the truncated backwards update in Section~\ref{subsec:ls:trunc strm} will match the performance of a full reconstruction in Algorithm~\ref{alg:Dynamic strm LS} almost exactly. 
\begin{figure*}[!ht]
	\centering
	\subfloat[]{\includegraphics[width=0.252\textwidth]{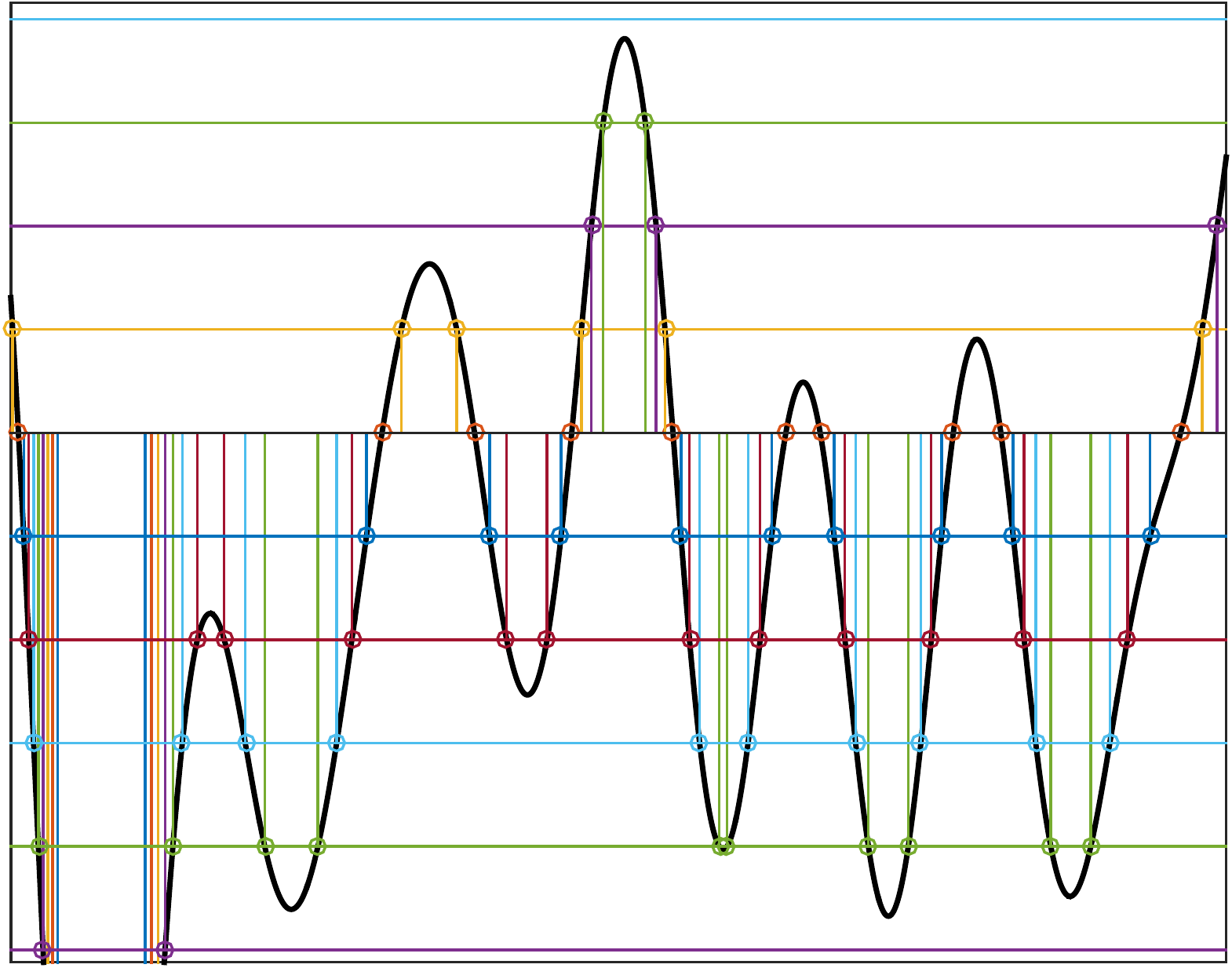}}
	\subfloat[k=4]{\includegraphics[width=0.25\textwidth]{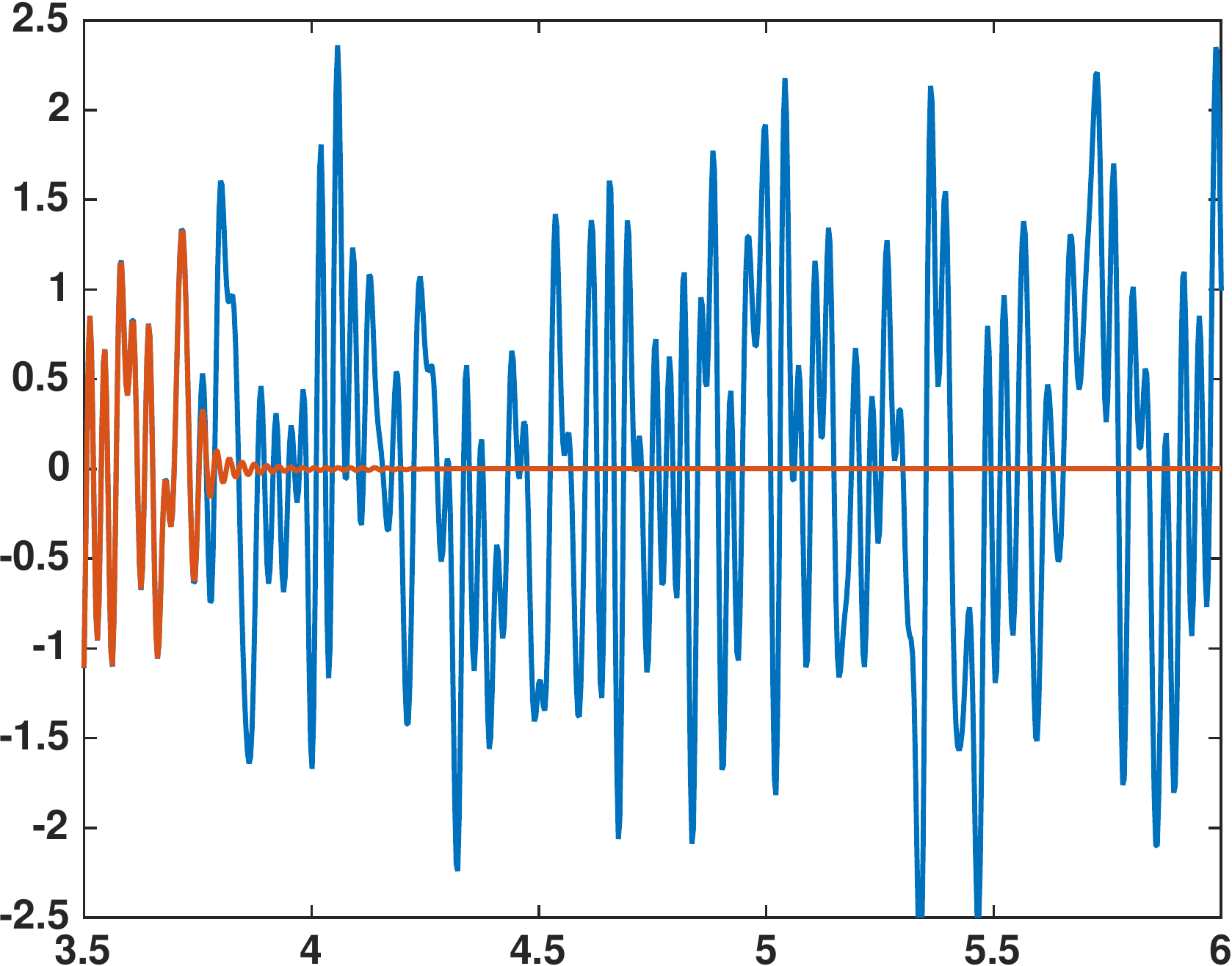}}
	\subfloat[k=5]{\includegraphics[width=0.25\textwidth]{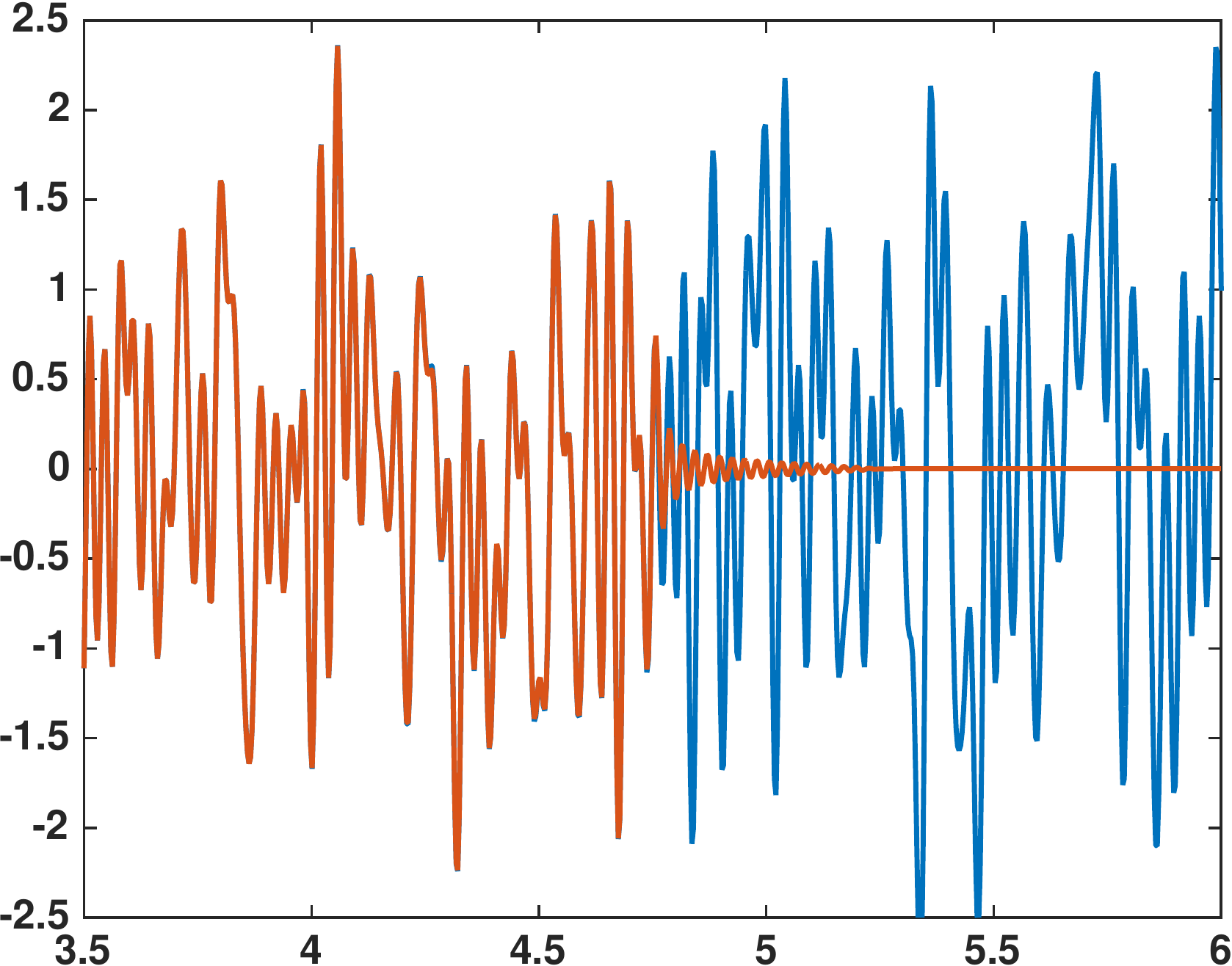}}
	\subfloat[k=6]{\includegraphics[width=0.25\textwidth]{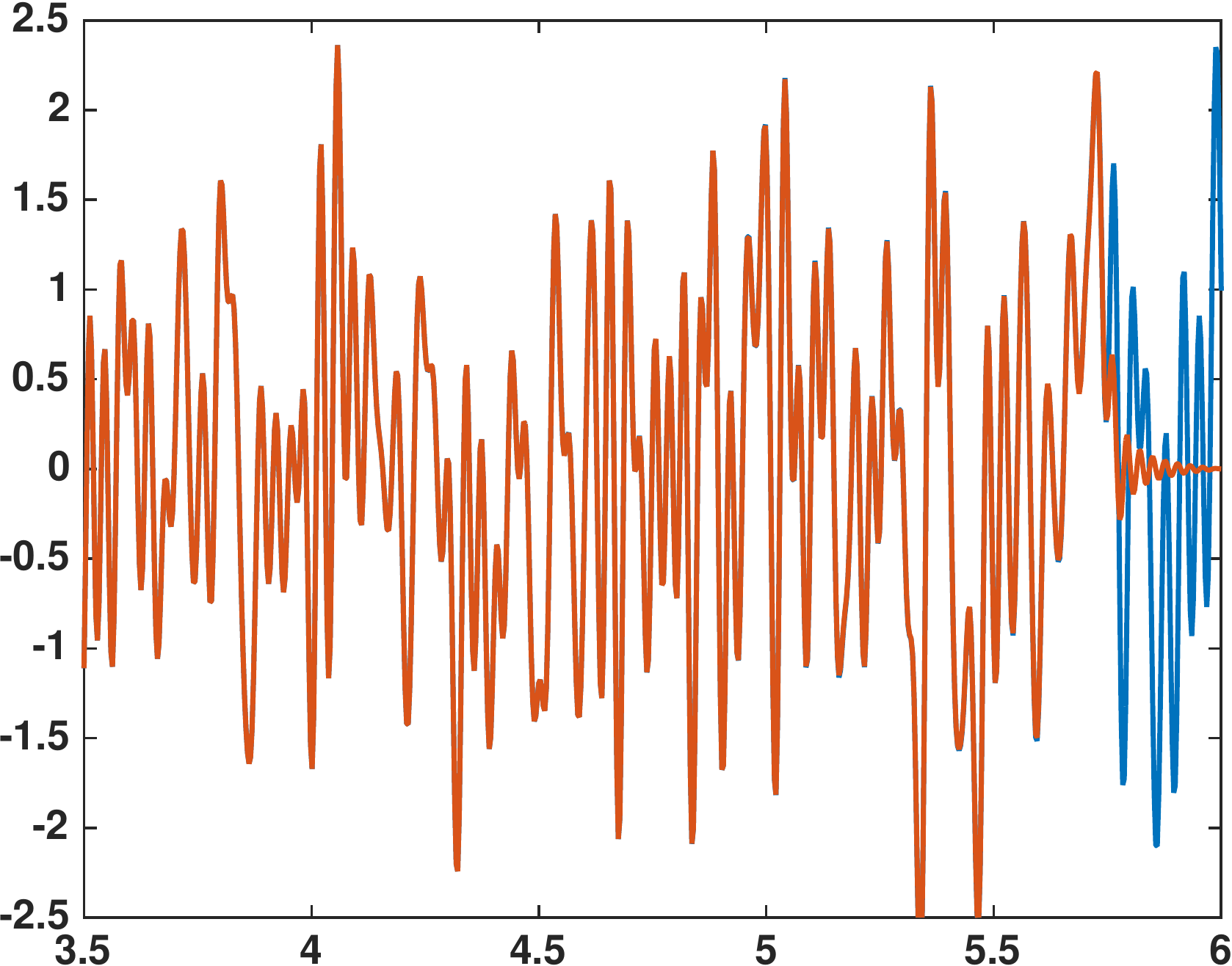}}
	\caption{(a) Level crossing samples; (b)-(d) The original signal (blue) and the reconstructed signal (orange) at time steps $k=4,5,6$.  
	}
	\label{fig:3figs}
\end{figure*}

\begin{table}[!hbt]
\begin{minipage}[b]{1.0\columnwidth}
	\caption{\small\sl The entries in the chart below tabulate how close our estimate of
	$\hat\vx_{j|k}$ is to the final least-squares estimate $\vx_{j}^*$ in a typical
	reconstruction problem.  The numbers below are
	$\log_{10}(\|\hat\vx_{j|k}-\vx_{j}^*\|/\|\vx_{j}^*\|)$.  The convergence is
	extremely rapid; after just three frames, we have achieved seven digits of accuracy.  This means
	that terminating the loop in Algorithm~\ref{alg:Dynamic strm LS} at $k=3$ costs us almost nothing in terms of
	reconstruction performance.
	}
	\setlength{\tabcolsep}{5.5pt}
	\begin{tabular}{c|cccccccc}
		& $j=4$ & $j=5$ & $j=6$ & $j=7$ & $j=8$ & $j=9$ & $j=10$\\\hline
		$k=4$ & -0.31 &  --- &  --- &  --- &  --- &  --- &  ---  \\ 
		$k=5$ & -3.39 & -0.32 &  --- &  --- &  --- &  --- &  ---  \\ 
		$k=6$ & -5.12 & -3.24 & -0.32 &  --- &  --- &  --- &  --- \\ 
		$k=7$ & -7.28 & -5.08 & -3.46 & -0.27 &  --- &  --- &  ---  \\ 
		$k=8$ & -9.27 & -7.08 & -5.60 & -3.44 & -0.34 &  --- &  ---  \\ 
		$k=9$ & -10.84 & -8.65 & -7.17 & -5.19 & -2.48 & -0.22 &  ---  \\ 
		$k=10$ & -13.27 & -11.08 & -9.60 & -7.62 & -4.90 & -3.44 & -0.36  \\ 
	\end{tabular}
	\label{tab:lagerror}
\end{minipage}

\end{table}

\subsection{Nonlinear regression of intensity function}
\label{subsec: nhpp numerical ex}
\noindent
In this example, we consider the problem of estimating neuron’s intensity function, which amounts to a nonlinear convex program: recover a signal from non-uniform samples of a non-homogeneous Poisson process (NHPP). Estimating the rate function is a fundamental problem in neuroscience \cite{Lewicki1998},
as the spikes' temporal pattern encodes information that characterizes the neuron's internal properties. A standard model for neural spiking under a rate model  \cite{rabinowitz2015attention,zhao2017variational,truccolo2005point} assumes that, given the rate function, the spiking observations follow a Poisson likelihood \cite{brown2002time}.

The formulation of the intensity function as optimization program is as follows. Given time series of spikes $\mathcal{H}_T =\{\tau_1,\dots,\tau_m\}$ in $[0,T]$, generated by $\mathcal{H}_T|\lambda(t)\sim \mathrm{Poisson}(\lambda(t))$, we want to estimate the underlying rate function $\lambda(t)$ using the maximum likelihood estimator
\cite{deneve2008bayesian, Brown2004, Brillinger1988}.
The associated optimization program is 
\begin{equation}\label{eq:spikesneglogobj}
		\hat \lambda(t)  
	 =  \arg\min_{\lambda(t)}   \int_{0}^T \lambda(t) dt - \sum_{i=1}^m \log(\lambda(\tau_i)).
\end{equation}

To make \eqref{eq:spikesneglogobj} well-posed, we need to introduce a model for $\lambda(t)$.  The straightforward one we use here is to write $\lambda$ as a superposition of basis functions, $ \lambda(t) = \sum_i \vx_i \psi_i(t)$.  Similar models which also fit into our streaming framework can be found in the literature.  For example,
\cite{ahmadi2018} uses Gaussian kernels while dynamically adapting the kernel bandwidth, \cite{dimatteo2001bayesian} uses splines for which the number and the location of knots are free parameters. Assuming a Gaussian prior on $ \lambda(t)$, \cite{solin2018infinite} divides the problem into small bins with each bin being a local Poisson and use Gaussian kernels.  Also related is the Reproducing Kernel Hilbert Space (RKHS) formulation in \cite{flaxman2017poisson}, %
and finally \cite{chakraborty2020sparse} shows that \eqref{eq:spikesneglogobj} can be  solved using a second-order variants of stochastic mirror descent based on ideas from \cite{flaxman2017poisson}.

In the example here we will consider parameterization of the intensity function using splines. Splines, in particular (cardinal) B-splines, have been successfully used to capture the relationship between the neural spiking and the underlying intensity function \cite{kass2003statistical, sarmashghi2021efficient}. B-splines have properties (see, e.g. \cite{de1966splines,de1978practical}) that make them favorable for this particular application. They are non-negative everywhere, and so restricting $ \vx $ to $ \R_+ $ guarantees $ \hat\lambda(t) \geq 0$. 
They have compact support (minimal support for a given smoothness degree), which nicely breaks the basis functions to overlapping frames. In addition, they are convenient to work with from the point of view of numerical analysis.
The results below were obtained with second-order B-splines, but higher models can be obtained in the same way.

 We  divide the time axis into short intervals (frames) and associate with each frame $N$ B-splines. The frames' length is set such that for each $t$,  $\lambda(t)$ is expressed by basis functions from at most two frames (see \cite{hamam2019second} for additional details). 

To cast \eqref{eq:spikesneglogobj} into the form of \eqref{eq:def:general problem formulation}, we define the local loss functions $ f_t$ as 
\begin{multline*}
	f_t\left(\vx_{t-1}, \vx_{t}\right) = 
	\left\langle\vx_{t}, \va_{t}\right\rangle + \left\langle\vx_{t-1}, \vb_{t}\right\rangle  \\
	- \sum_{m} \log \left(\langle\vx_{t}, \vc_{t,m}\rangle + 
	\langle\vx_{t-1}, \vd_{t,m}\rangle  \right),
\end{multline*}
where
$\vx_{k-1},\vx_k\in\R^N$ are the coefficient vectors in the $(k-1)$ and $k$ frames, and $\va_k,\vb_{k}\in\R^N$ are the basis functions from the same frames integrated over the $ k$th frame:
\[
\va_k = 
\int_{T_{k-1}}^{T_k}
\begin{bmatrix}
	\psi_{k,1}(t) \\
	\psi_{k,2}(t) \\
	\vdots \\
	\psi_{k,N}(t) 
\end{bmatrix} dt,
\quad
\vb_{k} = 
\int_{T_{k-1}}^{T_k}
\begin{bmatrix}
	\psi_{k-1,1}(t) \\
	\psi_{k-1,2}(t) \\
	\vdots \\
	\psi_{k-1,N}(t)
\end{bmatrix}dt,
\]
and $ \vc_{k,m} , \vd_{k,m} \in \R^N $ are samples of the basis functions at events observed during the $ k$th frame:
\[
\vc_{k,m} = 
\begin{bmatrix}
	\psi_{k,1}(\tau_m) \\
	\psi_{k,2}(\tau_m) \\
	\vdots \\
	\psi_{k,N}(\tau_m) 
\end{bmatrix},
\quad
\vd_{k,m} = 
\begin{bmatrix}
	\psi_{k-1,1}(\tau_m) \\
	\psi_{k-1,2}(\tau_m) \\
	\vdots \\
	\psi_{k-1,N}(\tau_m)
\end{bmatrix},
\]
with $M_k =\{ m: \tau_m\in [T_{k-1}, T_k) \} $.

The MLE optimization program can then be rewritten as  
\begin{equation}\label{eq:spikesfinaloptformu}
	 \underset{(\vx_0,\dots,\vx_{T})}{\arg\min}  \sum_{t=1}^T f\left(\vx_{t-1}, \vx_{t}\right), \quad \text{s.t. } \bvx_T \geq \mzero 
\end{equation}

The condition $\bvx_T \geq \mzero$ is set to ensure $ \hat\lambda(t)\geq 0 ~\forall t $.
We solve \eqref{eq:spikesfinaloptformu} with the NOA algorithm described in Section \ref{sec:nwton alg} with the following modification. We use log-barrier modification as in \cite[\S 11 ]{boyd2004convex}, since NOA was originally derived for unconstrained optimization problems. 

The simulation data was created by generating random smooth intensity function $ \lambda^*(t) $, which was then used to simulate events ("spikes") following non0-homogeneous Poisson distribution using the standard  thinning method \cite{lewis1979simulation}. Focusing on the effect of online estimation, we set the "ground truth" reference as the batch solution, to which we then compare the online estimates. 

The rapid convergence of the updates is depicted in Figure \ref{fig:res_time},
and in Figure \ref{fig:nhpp:buff error} we show the effect of truncating the updates on the solution accuracy.
\begin{figure}[!t]
	\centering
	\includegraphics[width=0.9\linewidth]{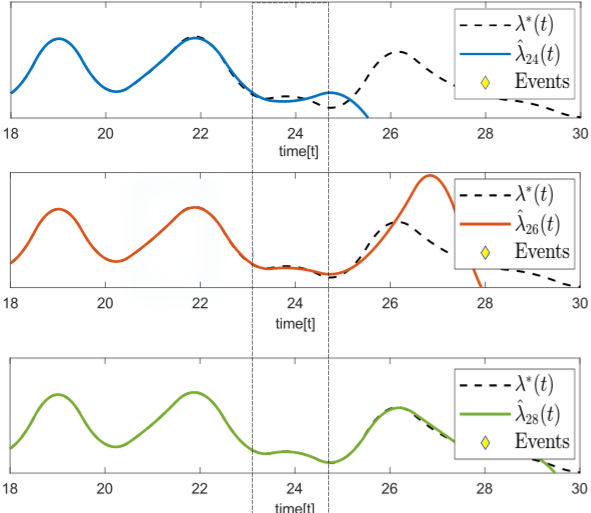}
	\caption{Illustration of the convergence of the corrections as they backpropagate. The figures show, from top to bottom,   $\lambda(t)$ at $T=24, 26, 28$. Following the dashed rectangles, moving down the plots shows the convergence to $\lambda^*(t)$ (the dashed line).}
	\label{fig:res_time}
\end{figure}

\begin{figure}[!t]
	\centering 
		\includegraphics[width=1.0\linewidth]{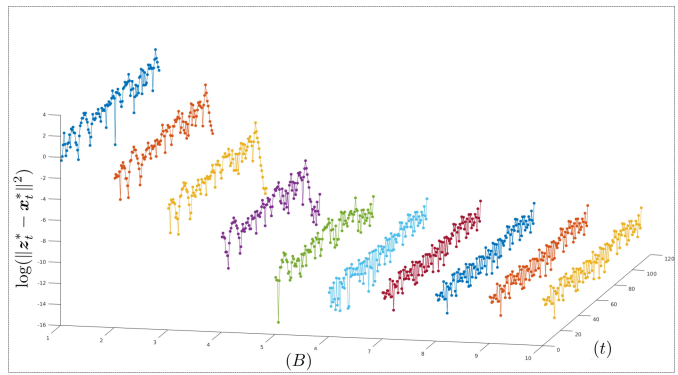}
	\caption{3D illustration of the truncation error for various buffer sizes. The (log) magnitude of the asymptotic error is plotted for different buffer sizes against time; for a fixed B, the error displays uniformly distributed in time without aggregation. It decreases exponentially as B increases, as predicted by the theoretical analysis. }
	\label{fig:nhpp:buff error}
\end{figure}

\section{Summary and future work}
\noindent
This paper focused on optimization problems expressed as a sum of compact convex loss functions that locally share variables. 
We have shown that the updates converge rapidly under mild conditions that correspond to the variables’ loose-coupling.
The main impact of the convergence result for the updates is that it allows us to approximate the solution via early truncating the updates with a negligible sacrifice of accuracy.
Driving these results, the primary underlying mechanism that resurfaced throughout was the block-tridiagonal structure of the Hessian. Rising through the structure of the derivatives of the loss functions, the block-tridiagonal structure led to efficient numerical algorithms and played a prominent role in the analysis and proofs to many results in this paper.  Future work includes extending these results to problems associated with more general variable dependency graphs and optimization programs with local constraints.


\appendix 


\subsection{Proof of Lemma~\ref{lm:xtbound}}
\label{sec:xtboundproof}
\noindent
Our proof uses the simple fact that a recursive set of inequalities of the form,
\begin{equation}\label{eq:contractiverecursiveseq}
		z_0 \leq b, \quad z_t\leq b + az_{t-1},\quad a,b\geq 0,
\end{equation}

will also obey, if $a<1$, 
\[
	z_t\leq b\left(\frac{1-a^{t+1}}{1-a}\right) \leq \frac{b}{1-a}.
\]
In the forward sweep in \eqref{eq:forwardsweep}, we have 
\begin{align*}
	\|\vv_0\|&= \|\mQ_0^{-1}\vg_0\|\leq M/(1-\eps_*), 
\end{align*}
and then for $t=1,\ldots,T-1$,
\begin{align*}
    \|\vv_t\|&= \|\mQ_t^{-1}(\vg_t-\mE_{t-1}\vv_{t-1})\| \leq M/(1-\eps_*) + \rho\|\vv_{t-1}\|\\
	& \Downarrow \\
	\|\vv_t\|&\leq \frac{M}{(1-\eps_*)(1-\rho)} =: M_v.
\end{align*}
For the backward sweep in \eqref{eq:backwardsweep}, we have
\begin{align*}
	\|\vx_T\| &= \|\vv_T\|  \leq M_v, 
\end{align*}
and then for $t=T-1,\ldots,0$,
\begin{align*}
	\|\vx_{t}\| &= \|\vv_t - \mQ_{t}^{-1}\mE_t^\T\vx_{t+1}\| \leq M_v + \rho\|\vx_{t+1}\| \\
	& \Downarrow \\
	\|\vx_t\| &\leq M_v\left(\frac{1-\rho^{T-t+1}}{1-\rho}\right)\\
	&= \frac{M(1-\rho^{T-t+1})}{(1-\eps_*)(1-\rho)^2}
	\leq \frac{M}{(1-\eps_*)(1-\rho)^2}.
\end{align*}

\subsection{Proof of Theorem~\ref{thm:ls convergence}}
\label{sec:lsconvproof}
\noindent
We start with a simple relation that connects the correction in frame bundle $t$ to the correction in
bundle $t-1$ as we move from measurement batch $T\rightarrow T+1$.  From the update
equations, we see that for $t\leq T$, 
\begin{equation*}
\hat\vx_{t-1|T+1} - \hat\vx_{t-1|T} =
-\mU_{t-1}\left(\hat\vx_{t|T+1}- \hat\vx_{t|T}\right), 
\end{equation*}
\begin{align*}
\|\hat\vx_{t-1|T+1} - \hat\vx_{t-1|T}\| &\leq \|\mU_{t-1}\|\cdot\|\hat\vx_{t|T+1}-
\hat\vx_{t|T}\| \\ &\leq \frac{\theta}{1-\epsilon_*}\cdot\|\hat\vx_{t|T+1}-
\hat\vx_{t|T}\|,
\end{align*} 
where we used the fact that $  \|\mU_{t-1}\| =  \| \mQ_{t-1}^{-1}\mE_{t-1}^T\|\leq(1-\epsilon_*)^{-1}\theta  $.
Applying this bound iteratively and using $t=T$, we can bound the correction error $\ell$ frames back as
\begin{equation}\label{eq:single step bnd}
 \|\hat\vx_{T-\ell|T+1} - \hat\vx_{T-\ell|T}\| ~\leq~
\left(\frac{\theta}{1-\epsilon_*}\right)^\ell\, \|\hat\vx_{T|T+1}- \hat\vx_{T|T}\|. 
\end{equation}
This says that the size of the update decreases {\em geometrically} as it back propagates through
previously estimated frames.  We can get a uniform bound on the size of the initial update $\|\hat\vx_{T|T+1}- \hat\vx_{T|T}\|$ by using Lemma~\ref{lm:xtbound} with
\[
    M = \max_t \left\|\begin{bmatrix} \mA_t^\T & \mB_{t+1}^\T \end{bmatrix} \begin{bmatrix} \vy_{t} \\
\vy_{t+1} \end{bmatrix} \right\| \leq \sqrt{1+\delta}\, M_y,
\]
to get
\begin{align*}
    \|\hat\vx_{T|T+1}- \hat\vx_{T|T}\| &\leq \|\hat\vx_{T|T+1}\| + \|\hat\vx_{T|T}\| \\
    &\leq \frac{M(2+\rho)}{(1-\eps_*)(1-\rho)} =: M_{\hat{x}}.
\end{align*}
Thus we have
\begin{equation*}
    \|\hat\vx_{t|T+1} - \hat\vx_{t|T}\| \leq
    M_{\hat{x}}\left(\frac{\theta}{1-\epsilon_*}\right)^{T-t}, \quad t\leq T,
\end{equation*}
and the $\{\hat\vx_{t|T}\}$ converge to some limit $\{\vx_{t}^*\}$.  

We can write the difference of the estimate $ \xhat_{t|T} $ from its limit point as the telescoping sum 
\begin{equation*}
    \hat\vx_{t|T} -   \vx_{t}^* = \sum_{\ell=0}^\infty \hat\vx_{t|T+\ell} -
    \hat\vx_{t|T+\ell+1}, 
\end{equation*}
and then using the triangle inequality, 
\begin{align}
\nonumber
\|\hat\vx_{t|T} -  \vx_{t}^*\| & \leq \sum_{\ell= 0}^\infty\|\hat\vx_{t|T+\ell} -
\hat\vx_{t|T+\ell+1}\|                                             \\
\nonumber                                     & \leq \sum_{\ell=
	0}^\infty\left(\frac{\theta}{1-\epsilon_*}\right)^{T+\ell-t}M_{\hat\alpha}                          \\
\label{eq:alphaconverge}                      & =	
M_{\hat{x}}\left(\frac{1-\epsilon_*}{1-\epsilon_*-\theta}\right)\left(\frac{\theta}{1-\epsilon_*}\right)^{T-t}.
\end{align}

\subsection{Proof of Lemma \ref{lemma:lip and conv of partially separable sum}}
\label{sec:aggliplemproof}
\noindent
We establish upper and lower bounds on the eigenvalues of the Hessian of $J_T$.  Let $ \mH = \nabla^2 J_T $ and\footnote{We assume that $\mG_t$ are zero padded to be the same size as $\mG$.} $ \mG_t = \nabla^2 f_t(\vx_{t-1}\,\vx_t)$, and consider the factorization $\mH = \mH^{o} + \mH^e $ with $\mH^{o}= \mG_1+\mG_3+\dots $ and $	\mH^{e}= \mG_2+\mG_4+\dots~$.
The fact that $f_t$ depends only on variables $ \vx_{t-1}$ and $ \vx_t$ implies that
$ \mH^{o} $ and $ \mH^{e} $ both are block diagonal matrices, and so both $\|\mH^o\|$ and $\|\mH^{e}\|$ have the easy upper bound $L_{\max}$.  It follows that $\|\mH\|\leq 2L_{\max}$.


%

For the lower bound, let $\bvy$ be an arbitrary block vector.  Then
\begin{alignat*}{2}
	\bvy^T \mH \bvy 
	&= \bvy^T \left(\sum_i \mG_i\right) \bvy 
	= \sum_i  \bvy^T \mG_i \bvy \\
	&= \sum_i  \begin{bmatrix} \vy_{i-1}^\T & \vy_i^\T \end{bmatrix}\mG_i \begin{bmatrix} \vy_{i-1} \\ \vy_i\end{bmatrix}\\
	& \geq \mu_{\min} \sum_i \|\begin{bmatrix} \vy_{i-1} \\ \vy_i\end{bmatrix}\|^2\\
	& \geq \mu_{\min} \sum_i \|\vy_i\|^2 = \mu_{\min}\|\bvy\|^2.
\end{alignat*}
Thus the smallest eigenvalue of $\mH$ is at least $\mu_{\min}$.

\subsection{Proof of Theorem~\ref{thm:convergence of the updates}}
\label{sec:weaker conv result proof}
\noindent
Recall that $\xhat_{t|T-1}$ is frame $t$ of the minimizer to $J_{T-1}$, and $\xhat_{t|T}$ is frame $t$ of the minimizer to $J_{T}$.  We will start by upper bounding how much the solution in frame $t$ moves as we transition from the solution of $J_{T-1}$ to the solution of $J_T$.  We bound this quantity, $ \| \xhat_{t|T} - \xhat_{t|T-1}\|$ by tracing the steps in the gradient descent algorithm for minimizing $J_T$ when initialized at the minimizer of $J_{T-1}$ and an appropriate choice for the new variables introduced in frame $T$.  Next, we give a simple argument that the $\{\hat\vx_{t|T} \}_{t\leq T}$ form a Cauchy (and hence convergent) sequence.

To avoid confusion in the notation, we will use $\ybar\nw{k} = \{\vy\nw{k}_t\}_{t=0}^{T}$ to denote the gradient descent iterates.
We initialize $\ybar\nw{0}$ with the $\{\hat\vx_{t|T-1}\}_t$ and a $\vw_T$ that obeys \eqref{eq:unigradbound}
\begin{equation}
	\label{eq:gdinitw}
	\vy\nw{0}_t := 
	\begin{cases}
		\xhat_{t|T-1}, & t=0,\dots,T-1\\
		\vw_T, & t=T,
	\end{cases}
\end{equation} 
and iterate using 
\begin{equation*}
	\ybar\nw{k+1}  \gets  \ybar\nw{k}+\bvd\nw{k}, \quad
	\bvd\nw{k}  \triangleq  -h \nabla J_T(\ybar\nw{k}).
\end{equation*}
We know that for an appropriate choice of stepsize $h$ (this will be discussed shortly), the iterations converge to $  \bvy^*=\hat\bvx_{T}$.  We also know that since $J_T$ is strongly convex and has a Lipschitz gradient (recall Lemma~\ref{lemma:lip and conv of partially separable sum}), this convergence is linear (see, for example, \cite[Thm 2.1.15]{nesterov2018lectures}):
\begin{equation}
	\label{eq:gd convg rate}
	\|\bvy\nw{k}-\bvy^*\|\leq   r_0 a^{k},\quad\text{where}~~
	a = \left(\frac{2 L_{\max} - \mu_{\min} }{2 L_{\max} + \mu_{\min}} \right),
\end{equation}
and $r_{0} = \|\bvy^* - \bvy\nw{0}\|$.

The key fact is that the $\bvd\nw{k}$, which are proportional to the gradient, are highly structured.  Using the notation $\nabla_t$ to mean ``gradient with respect to the variables in frame $t$'', we can write $\nabla J_T$ in block form as
\[
	\nabla J_T(\bvy\nw{k}) = 
	\medmath{
	\begin{bmatrix}
		\nabla_0 f_1(\vy\nw{k}_0,\vy\nw{k}_1) \\
		\nabla_1 \left( f_1(\vy\nw{k}_0,\vy\nw{k}_1) +  f_2(\vy\nw{k}_1,\vy\nw{k}_2) \right)\\
		\vdots \\
		\nabla_{T-1}\left( f_{T-1}(\vy\nw{k}_{T-2},\vy\nw{k}_{T-1}) + f_T(\vy\nw{k}_{T-1},\vy\nw{k}_T) \right)\\
		\nabla_T f_T(\vy\nw{k}_{T-1},\vy\nw{k}_T)
	\end{bmatrix}}.
\]
Because of the initialization \eqref{eq:gdinitw}, the first step $\bvd\nw{0}$  is zero in every single block location except the last two
\[
	\bvd\nw{0} = \begin{bmatrix}
		0& 
		\hdots &
		0&
		*&
		*
	\end{bmatrix}^T.
\]
As such, only the variables in these last two blocks change; we will have $\vy\nw{1}_t = \vy\nw{0}_t$ for $t=0,\ldots,T-2$.  Now we can see that $\nabla J_T(\ybar\nw{1})$ (and hence $\bvd\nw{1}$) is nonzero only in the last three blocks.  Propagating this effect backward, we can say that for any $1 \leq \tau\leq T$
\[
	\vd\nw{k}_{T-\tau} = \vzero,\quad\text{for}~~k=0,\ldots,\tau-1.
\]
Thus we have
\begin{multline*}
		\|\xhat_{T-\tau|T} - \xhat_{T-\tau|T-1} \| = \|\sum_{k=\tau}^\infty\vd\nw{k}_{T-\tau} \| 
	= \|\vy\nw{\tau}_{T-\tau}-\vy^*_{T-\tau}\| 
	\\	
	\leq \|\ybar\nw{\tau}-\ybar^* \| 
	\leq r_0 a^\tau.
\end{multline*}

To bound $ r_0$ we use the $\tilde\mu$ strong convexity of $J_T$ from Lemma~\ref{lemma:lip and conv of partially separable sum} and the assumption \eqref{eq:unigradbound},
\begin{multline*}
    r_0 = \|\ybar^* - \ybar\nw{0}\| 
    \leq \dfrac{1}{\tilde \mu} \|\nabla J_T(\ybar\nw{0} )\|
     \\
	= \dfrac{1}{\tilde \mu}\|\nabla f_T(\hat{\vx}_{T-1|T-1},\vw_T)\| 
	\leq {\tilde\mu}^{-1}  M_g,
\end{multline*}
and so 
\begin{equation}\label{eq:expon decaying update}
	\|\xhat_{t|T} - \xhat_{t|T-1} \|  
	\leq C_0 a^{T-t}, \quad  t < T ,\quad C_0:={\tilde\mu}^{-1}  M_g
\end{equation}
An immediate consequence of \eqref{eq:expon decaying update} is that $\{\hat\vx_{t|T}\}_T$  is a Cauchy sequence.  For all $ t\geq 0  $ and $n>l>0$  we have that
\begin{alignat*}{2}
	&\|{\xhat_{t|t+n+l}-\xhat_{t|t+n} }\| 
	=  \| {\sum_{k=1}^{l} \xhat_{t|t+n+k}-\xhat_{t|t+n+k-1}}\|\\
	&\quad  \leq \sum_{k=1}^{l} \|{\xhat_{t|t+n+k}-\xhat_{t|t+n+k-1}}\| 
	\leq C_0 \sum_{k=1}^{l} a^{n+k}\\
	&\hspace{2cm}  = C_0\left( \dfrac{1-a^l}{1-a} \right) a^{n+1},
\end{alignat*}
and so
\[
	\lim\limits_{n,l\rightarrow\infty} \norm{\xhat_{t|t+n+l}-\xhat_{t|t+n} } = 0,
\]	
and $\left\{\xhat_{t|T} \right\}_T$ has a limit point that is well defined \cite{royden1988real}, 
\[
	\vx_t^*  = \lim\limits_{T \rightarrow \infty}	\xhat_{t|T}.
\]
By taking $n=T-t $, and taking the limit as $ l\rightarrow\infty $ we have
\begin{alignat}{2}\label{eq:temp bnd}
	\|\xhat_{t|T} -\vx_{t}^* \|  \leq
	& C_0  \left( \dfrac{a}{1-a} \right) a^{T-t}.
\end{alignat} 
Plugging in our expression for $a$ yields
\begin{equation}
	\label{eq:final weak converg up-bnd}
	\left\|{\vx_{t}^* - \xhat_{t|T} }\right\|  \leq
	C_1\cdot  \left(\frac{2 L_{\max} - \mu_{\min} }{2 L_{\max} + \mu_{\min}} \right)^{T-t},
\end{equation}
with $ C_1 :=  C_0 \left( \dfrac{2 L_{\max} -\mu_{\min}}{2 \mu_{\min}} \right)$.

\subsection{Proof of  Theorem~\ref{thm:bondcnvgrad}}
\label{sec:cnvxconvrgproof}
\noindent
We being by recalling the gradient theorem, which states that for a twice differentiable function
\begin{equation}\label{eq:gradthm}
		\nabla f(\vy)  = \nabla f(\vx)  + \left(\int_0^1 \nabla^2 f(\vx + \tau(\vy-\vx ))d\tau\right) (\vy-\vx ),
\end{equation}
for any $\vx,\vy$.  In particular, since $\nabla f_t(\bar\vx_{t-1|t},\bar\vx_{t|t}) = \vzero$, we can write

\begin{equation*}\label{eq:lineintHesdef}
		\nabla f_t(\hat\vx_{t-1|T},\hat\vx_{t|T}) = 
	\begin{bmatrix} \mG_{t-1,t} & \mE_t^\T \\ \mE_t & \mG_{t,t}\end{bmatrix}
	\begin{bmatrix}\hat\vx_{t-1|T}-\bar\vx_{t-1|t} \\ \hat\vx_{t|T}-\bar\vx_{t|t} \end{bmatrix},
\end{equation*}

where
\begin{equation}\label{eq:lineintHes}
		\medmath{
	\begin{bmatrix} \mG_{t-1,t} & \mE_t^\T \\ \mE_t & \mG_{t,t}\end{bmatrix} = 
	\int_{0}^1 \nabla^2 f_t\left(\begin{bmatrix} \bar\vx_{t-1|t}+\tau(\hat\vx_{t-1|T}-\bar\vx_{t-1|t}) \\ \bar\vx_{t|t}+\tau(\hat\vx_{t|T}-\bar\vx_{t|t})  \end{bmatrix} \right)~d\tau
	},
\end{equation}
and we have used the stacked notation $f_t\left(\begin{bmatrix} \vu\\\vv\end{bmatrix}\right)$ in place of $f_t(\vu,\vv)$ for convenience.  With this notation, we can rewrite the optimality condition
\begin{align*}
	& \nabla J_T = \\
	& ~~ \medmath{
		\begin{bmatrix}
		\nabla_0 f_1(\hat\vx_{0|t},\hat\vx_{1|T}) \\
		\nabla_1 f_1(\hat\vx_{0|t},\hat\vx_{1|T}) + \nabla_1 f_2(\hat\vx_{1|T},\hat\vx_{2|T}) \\
		\vdots \\
		\nabla_{T-1}f_{T-1}(\hat\vx_{T-2|T},\hat\vx_{T-1|T}) + \nabla_{T-1}f_T(\hat\vx_{T-1|T},\hat\vx_{T|T}) \\
		\nabla_T f_T(\hat\vx_{T-1|T},\hat\vx_{T|T})
	\end{bmatrix}
	} 
    = \vzero
\end{align*}
as a block-tridiagonal system with the  same form as \eqref{eq:blktridiagsys} where

\begin{equation}\label{eq:lineintHtdef}
	\mH_t = 
\begin{cases}
	\mG_{0,1},& t=0;\\
	\mG_{t,t}+\mG_{t,t+1},& t=1,\dots,T-1\\
	\mG_{T,T},&t=T;
\end{cases}
\end{equation}

and 
\[
\vg_t= 
\begin{cases}
	\medmath{
		\begin{bmatrix}
			\mG_{0,1} & \mE_{0}^{\top}
		\end{bmatrix}
		\begin{bmatrix}
			\bar{\vx}_{0|1} \\
			\bar{\vx}_{1|1}
	\end{bmatrix}} , & t=0; \\
	\medmath{
		\begin{bmatrix}
			\mE_{t-1}&  \mG_{t,t} 
		\end{bmatrix}
		\begin{bmatrix}
			\bar{\vx}_{t-1|t} \\
			\bar{\vx}_{t|t} 
	\end{bmatrix}}\\
	+ 
	\medmath{
		\begin{bmatrix}
			\mG_{t,t+1}&  \mE_t^T
		\end{bmatrix}
		\begin{bmatrix}
			\bar{\vx}_{t|t+1} \\
			\bar{\vx}_{t+1|t+1} 
		\end{bmatrix}
	} &t=1,\dots T-1\\
	\medmath{
		\begin{bmatrix}
			\mE_{T-1}&  \mG_{T,T} 
		\end{bmatrix}
		\begin{bmatrix}
			\bar{\vx}_{T-1|T} \\
			\bar{\vx}_{T|T}
	\end{bmatrix}} & t=T.
\end{cases}
\]
	Since the $ f_t $ are strongly convex, we have that $\mu_{min}\leq \|G_{i,j}\|\leq L_{\max}$ for all $i$ and $ j=i,i+1$. Hence, the main diagonal blocks $ \mH_t $ satisfy $\|\kappa^{-1}\mH_t-\mId\|\leq \delta$ with $ \kappa = (2L_{\max}+\mu_{\min})/2$ and $ \delta =({2L_{\max}-\mu_{\min}})/({2L_{\max}+\mu_{\min}} )$. Defining $\theta$ as the smallest upper bound such that $ \|\mE_t\|\leq \kappa\theta$,
we have 
\[
	\|\begin{bmatrix} \mG_{t-1,t} & \mE_{t-1} \end{bmatrix} \| \leq \kappa\sqrt{L_{\max}^2+\theta^2}
\]  
and can use the same bound for $\|\begin{bmatrix}\mE_{t-1} & \mG_{t,t} \end{bmatrix}\|$, and so 
\[
	\|\vg_t\| ~\leq~ 2M_x\kappa\sqrt{L_{\max}^2+\theta^2} =: M_g,
\]
for all $t$.  Then if $ \|\mE_t\| \leq \mu_{\min}/2 $, by Lemma~\ref{lm:Qcond} there will be an $\eps_*$ such that $\|\mQ_t^{-1}\|\leq 1/(1-\eps_*)$ for all $t$, satisfying $\rho= \theta/(1-\eps_*)<1$, and the result follows by applying Lemma~\ref{lm:xtbound}.


\subsection{Proof of Lemma \ref{lem:truncdepbndtake2}}
\label{sec:lemtruncupdateproof}
\noindent
The optimality of $\{\xfhat_{t|T}\} $ implies $ \nabla J_{T-B}(\underline\xfhat_{[T-B:T|T]})$ is all zeros except for the first term. Proposition \ref{prop:solvWfixVar} implies the same for $ \nabla J_{T-B}(\xbarhat_{[T-B:T|T]})$. Applying \eqref{eq:gradthm} to $ J_{T-B}$ with $ \vy:= \underline\xhat_{[T-B:T|T]}) $ and $ \vx := \underline\xfhat_{[T-B:T|T]}) $ following the same process from the proof of Theorem \ref{thm:convergence of the updates}, we get

\begin{multline*}
	\label{eq:sparsebt3dsys}
	\left[
		\begin{array}{c|cccc}
			\mH'_{T-B}   & \mE_{T-B}^T& \      & \         &  \\ \hline
			\mE_{T-B} & \mH_{T-B+1} & \mE_{T-B+1}& \         &  \\
			\     & \          & \      & \          &  \\
			\     & \         &  \  \ddots    & \         &  \\
			\ &  \     &  \   &  &       \\
			\     & \         & \       \mE_{T-1} & \mH_T
		\end{array}
		\right]\\
		\times \left[\begin{array}{c}
			\xfhat_{T-B|T}-\xhat_{T-B|T}   \\ 
			\hline 	\xfhat_{T-B+1|T}-\xhat_{T-B+1|T}   \\ \\ \vdots    \  \\ \\\zhat_{T|T}-\xhat_{T|T} 
		\end{array}\right] 
		=
		\left[\begin{array}{c}
			\vq_0\\
			\hline  0  \\ \\ \vdots     \ \\  \\0
		\end{array}\right],
\end{multline*}
with $ \mH_{T-B+i},i=1,\dots,B $ as in \eqref{eq:lineintHtdef}, and 
$ \mH'_{T-B} = \mG_{T-B|T-B+1}$, and 
$
\vq_0  =  \nabla_{T-B} J_{T-B}(\underline\xfhat_{[T-B:T|T]})  -\nabla_{T-B} J_{T-B}(\underline\xfhat_{[T-B:T|T]}).
$
Applying Lemma \ref{lem:sparseb3dsystake2} with $ \alpha = \theta $ and $ \beta = (1-\delta)^{-1}$ gives the Lemma's result.

\subsection{Proof of Theorem~\ref{thm:cnvxbufferrbnd}}
\label{sec:trunccnvxproof}
\noindent
The proof follows by breaking the error term into two parts: a self error term due to early termination of the updates and a bias term due to errors in preceding terms. The argument then follows by expressing the bias error as a convergent recursive sequence.
	
Note that  $\xfstar_t = \xfhat_{t|t+B-1} $ --- the last round at which it was updated. Adding and subtracting $ \xfhat_{t|T+B-1}$, we have
\[
		\|\vx^*_t-\xfstar_t\|  \leq 
	\underbrace{\|\xhat_{t|t+B-1} - \xfhat_{t|t+B-1} \|}_{e_b} + \underbrace{\| \vx^*_t -  \xhat_{t|t+B-1}\| }_{e_f}
\]

For $ e_f $, we have the easy bound (cf. Theorem \eqref{thm:convergence of the updates})
\[	e_f = \| \vx^*_t -  \xhat_{t|t+B-1} \| \leq \cnvxcnvgres{B-1}.\]
	
Invoking Lemma  \ref{lem:truncdepbndtake2} with $ T=t+B-1 $, we have that 
\begin{alignat*}{2}
 \|\xhat_{t|t+B-1}- \xfhat_{t|t+B-1}\| 
 	& \overset{(1)}{\leq}\frac{\theta}{(1-\delta)}\|\xhat_{t-1|t+B-1}-\xfstar_{t-1}\|
	\\
	&		 \hspace{-1.0cm}\overset{(2)}{=} \frac{\theta}{(1-\delta)}\|\xhat_{t-1|t+B-1}-\xfhat_{t-1|t+B-2}\|\\
	&\hspace{-1.00cm}\overset{(3)}{\leq} \frac{\theta}{(1-\delta)} \| \xhat_{t-1|t+B-1} -   \xhat_{t-1|t+B-2}\| \\
	&\hspace{-1.0cm}+\frac{\theta}{(1-\delta)} \|\xhat_{t-1|t+B-2}  - \xfhat_{t-1|t+B-2} \|, 
\end{alignat*}
where $\overset{(1)}{\leq}$ a is direct result of Lemma  \ref{lem:truncdepbndtake2}, $ \overset{(2)}{=} $ holds because $\xfstar_{t-1}$ is the same as the solution $ \xfhat_{t-1|t+B-2}$, and in $ \overset{(3)}{\leq}$, we add and subtract $ \xhat_{t-1|t+B-2} $, and then apply the triangle inequity. 

From \eqref{eq:expon decaying update}, we know that
\[
\| \xhat_{t-1|t+B-1} -\xhat_{t-1|t+B-2} \| \leq \singupdatebound{B} .
\]
Defining
$r(t)\triangleq \|\xhat_{t|t+B-1}-\xfstar_{t}\|$,
we obtain the recursive relation
\begin{equation*}
	r(t) \leq \singupdatebound{B} + \frac{\theta}{(1-\delta)} r(t-1) =\singupdatebound{B} + \frac{\theta}{(1-\delta)} r(t-1) .
\end{equation*}
The geometric series, using \eqref{eq:contractiverecursiveseq}, converges to
\begin{equation*}
	e_b  \leq  \singupdatebound{B}\frac{1}{(1-\theta(1-\delta)^{-1})}.
\end{equation*}
Lastly, combining $ e_f$ and $e_b $ yields
\begin{equation*}
	\|\vx_t^* - \xfstar_t\|  \leq 
	C_b \left(\dfrac{2L_{\max} -\mu_{\min}}{2L_{\max} +\mu_{\min}}\right)^{B},
\end{equation*}
with,
\[
	C_b := C_0\left(\frac{1}{1-a}+\frac{1}{(1-\theta(1-\delta)^{-1})} \right).
\]


\bibliographystyle{./IEEEtran}
\bibliography{IEEEabrv,refs}

	
\end{document}